\documentclass[12pt]{article}
\usepackage{amsmath}
\usepackage{amssymb}
\usepackage{mathrsfs}
\usepackage{amsthm}
\usepackage{rotating} 
\usepackage{authblk}
\usepackage{booktabs, longtable, array}
\usepackage{color}
\usepackage{graphicx}
\usepackage{caption}
\usepackage{subcaption}
\usepackage{soul}
\usepackage{bbm} 
\usepackage{ulem}
\usepackage{cancel} 
\usepackage{algorithm}
\usepackage{algorithmicx}
\usepackage{amssymb}
\usepackage{threeparttable}
\usepackage{titlesec}
\usepackage{tikz}
\usepackage{bbding}
\usepackage{pifont}
\usepackage{titletoc}
\usepackage{hhline}

\theoremstyle{definition}

\newtheorem{fact}{Fact}

\newtheorem{definition}{Definition}
\newtheorem{theorem}{Theorem}
\newtheorem{lemma}{Lemma}
 
\newtheorem{remark}{Remark}
\newtheorem{proposition}{Proposition}
\usepackage{hyperref}
\usepackage[margin=1.2in]{geometry}

\usepackage{hyperref}

\newcommand{\va}{\pmb{a}}
\newcommand{\vb}{\pmb{b}}
\newcommand{\vw}{\pmb{w}}

\newcommand{\x}{\pmb{x}}
\newcommand{\y}{\pmb{y}}
\newcommand{\z}{\pmb{z}}
\newcommand{\M}{\pmb{M}}
\newcommand{\X}{\pmb{X}}

\newcommand{\xu}{\pmb{u}}

\newcommand{\E}{\mathbb{E}}

\newcommand{\R}{\mathbb{R}}
\newcommand{\C}{\mathbb{C}}

\newcommand{\mA}{\mathcal{A}}
\newcommand{\mT}{\mathcal{T}}

\newcommand{\mF}{\mathcal{F}}

\def\lV{\left\lVert}
\def\rV{\right\lVert}
\def\lv{\left\lvert}
\def\rv{\right\lvert}
\def\lk{\left(}
\def\rk{\right)}
\def\lg{\langle}
\def\rg{\rangle}
\def\lz{\left[}
\def\rz{\right]}

\begin{document}

\title{Low-Rank Matrix Recovery via Heavy-Tailed Quadratic Sampling\footnote{This work was supported by
NSFC under grant number U21A20426}}

\author{Gao Huang\footnote{School of Mathematical Science, Zhejiang University, Hangzhou 310027, P. R. China, E-mail address: hgmath@zju.edu.cn}}	
\author{Song Li\footnote{School of Mathematical Science, Zhejiang University, Hangzhou 310027, P. R. China, E-mail address: songli@zju.edu.cn}}
\date{}
\affil{}
\renewcommand*{\Affilfont}{\small\itshape}
	
	\maketitle
	\begin{abstract}	
The problem of recovering an (approximately) low-rank Hermitian matrix $\M_0 \in \mathbb{C}^{n \times n}$ of rank $r$ from quadratic sampling matrices of the form $\{\va_k \va_k^*\}_{k=1}^m$ arises in a variety of applications, including phase retrieval. 
To obtain rigorous recovery guarantees, the sampling vectors $\{\va_k\}_{k=1}^m$ are typically modeled probabilistically. 
However, most existing theoretical results rely on Gaussian or sub-Gaussian assumptions, which may not accurately capture practical data models. 
In many applications, sampling vectors exhibit heavier tails, while theoretical understanding in such regimes remains scarce.

In this paper, we bridge this gap. We show that two widely used convex approaches, nuclear norm minimization and semidefinite-constrained empirical risk minimization, achieve uniform, stable, and robust recovery under the mild assumption that the entries of the sampling vectors have only finite $4+\delta$ moments, with the optimal sample complexity $m = \mathcal{O}\lk rn\rk$ up to moment-dependent constants.
The two main ingredients of our analysis are moment estimates for quadratic forms established via decoupling, together with recent advances in covariance estimation in heavy-tailed settings. 
As byproducts, we also establish the optimal sample complexity for low-rank matrix recovery under complex projective 
$4$-design sampling, thereby improving upon previous results, and obtain stability guarantees for phase retrieval under similarly weak moment assumptions.
      	\end{abstract}		

\section{Introduction}

The problem of recovering a low-rank matrix from a small number of linear measurements is a central topic in applied mathematics, statistics, electrical engineering, and computer science; see, e.g.,~\cite{recht2010guaranteed,chandrasekaran2012convex}. 
It arises in a variety of areas, including quantum tomography~\cite{gross2010quantum,flammia2012quantum,kueng2017low}, signal processing~\cite{ahmed2014compressive}, recommender systems~\cite{koren2009matrix}, and linear system identification and control~\cite{liu2010interior}. 
A prominent example is phase retrieval, which arises in a range of signal and imaging applications, including X-ray crystallography, astronomical imaging, and diffraction imaging~\cite{millane1990phase,shechtman2015phase}. 
In phase retrieval, the apparent obstacle posed by nonlinear magnitude-only measurements can be overcome by lifting the problem to a matrix space, an idea first introduced by Balan et al.~\cite{balan2009painless}.
This viewpoint later inspired the \textit{PhaseLift} approach of Candès et al.~\cite{candes2015phase1,candes2013phaselift}, which recasts phase retrieval as a low-rank matrix recovery problem.

Motivated by these applications and the close connection with phase retrieval, 
in this paper we study the recovery of an (approximately) low-rank Hermitian matrix $\M_0\in\mathcal{H}_n$ from the quadratic (i.e., rank-one) sampling model
\begin{equation}\label{eq:measurement}
    y_k=\left\langle \va_k\va_k^*,\M_0\right\rangle+\omega_k,\qquad k=1,\ldots,m.
\end{equation}
Here, $\mathcal{H}_n$ denotes the space of $n\times n$ complex Hermitian matrices, 
$\left\{\va_k\right\}_{k=1}^m$ are the sampling vectors, 
$\y:=\left\{y_k\right\}_{k=1}^m$ denotes the measurement vector 
and $\pmb{\omega}:=\left\{\omega_k\right\}_{k=1}^m$ denotes the measurement noise. 
When $\M_0=\x_0\x_0^*$ is rank-one for some $\x_0\in\C^n$, \eqref{eq:measurement} reduces to the intensity-only measurement model arising in phase retrieval~\cite{balan2009painless,candes2015phase1}.
To describe the setup more precisely, let $\mA:\mathcal{H}_n\to\R^m$ denote the linear map
\begin{equation}
    \mA\lk\M\rk=\left\{\left\langle \va_k\va_k^*,\M\right\rangle\right\}_{k=1}^m.
\end{equation}
Then \eqref{eq:measurement} can be written compactly as
\begin{equation}\label{eq:sample}
    \y=\mA\lk\M_0\rk+\pmb{\omega}.
\end{equation}

A prominent approach for recovering the matrix $\M_0$ from~\eqref{eq:sample} is nuclear norm minimization, formulated as the following convex program~\cite{candes2013phaselift,chen2015exact,cai2015rop,kueng2017low,kabanava2016stable,huang2025low,gilles2025stable}:
 \begin{equation}\label{nuclear}
\min_{\M\in\mathcal{H}_n} \,\lV\M\rV_{*}\quad\text{subject to}  \quad\lV\mA\lk\M\rk-\y\rV_{\ell_q}\le \eta,	 
\end{equation}
    where $\lV \M\rV_*$ denotes the nuclear norm of $\M\in\mathbb{C}^{n\times n}$, and $\eta$ is a known upper bound on the noise level, namely, $\lV \pmb{\omega}\rV_{\ell_q}\le \eta$.
    Here, for a vector $\x$, $\lV \x\rV_{\ell_q}$ denotes the usual $\ell_q$-norm.
In some situations, it is known a priori that the target matrix $\M_0$ is Hermitian positive semidefinite, i.e., $\M_0\succeq \pmb{0}$. 
In this case, one may replace the nuclear norm minimization program~\eqref{nuclear} with the semidefinite-constrained empirical risk minimization program~\cite{demanet2014stable,candes2014solving,krahmer2020complex,krahmer2022robustness,huang2025stable}:
\begin{equation}\label{PSD}
\min_{\M\in\mathcal{H}_n} \,\lV\mA\lk\M\rk-\y\rV_{\ell_q}\quad\text{subject to}  \quad\M\succeq \pmb{0},			 
\end{equation}
which is noise-blind in the sense that it does not require prior knowledge of the noise level.
Beyond these convex formulations, related algorithmic developments for the quadratic sampling model~\eqref{eq:measurement} include nonconvex matrix factorization methods~\cite{li2021nonconvex}, stochastic gradient-type algorithms~\cite{qin2024general}, hard-thresholding-type methods~\cite{foucart2019iterative,eisenmann2023riemannian}, and recent analyses of the nonconvex landscape of related formulations~\cite{mcrae2025nonconvex}.

From a theoretical perspective, the analysis of the quadratic sampling model~\eqref{eq:measurement} has so far relied predominantly on probabilistic assumptions on the sampling vectors.
Most existing results are established for Gaussian, or more generally sub-Gaussian, ensembles, both in low-rank matrix recovery and in the rank-one special case of phase retrieval~\cite{candes2013phaselift,candes2014solving,chen2015exact,cai2015rop,kueng2017low,kabanava2016stable,li2021nonconvex,krahmer2020complex,gao2021phase,kim2024robust,maunu2024acceleration,huang2026robust,huang2025stable,mcrae2025nonconvex}, or on structured sampling models such as coded diffraction patterns, convolutional ensembles and projective $t$-design ensembles~\cite{candes2015phase,qu2020convolutional,kueng2017low,li2025truncated}.
However, these idealized probabilistic models cannot fully characterize the sampling patterns arising in practical acquisition systems, and even structured models, though closer to certain acquisition mechanisms, do not encompass all sampling ensembles encountered in practice.
For instance, in practical ghost imaging, illumination patterns are often generated by experimentally constrained optical architectures and, in modalities such as X-rays, electrons, and neutrons, are commonly realized by transversely translating a fixed mask~\cite{kingston2023optimizing,aminzadeh2023mask}, leading to sampling ensembles that may deviate substantially from ideal Gaussian ensembles and may even exhibit heavy-tailed behavior.
These considerations motivate the study of recovery guarantees under substantially weaker distributional assumptions on the sampling vectors, which naturally leads to the following question:
\begin{itemize}
\item[] \textbf{\textit{Can one achieve recovery of a low-rank matrix from the quadratic sampling model~\eqref{eq:measurement} when the sampling vectors are heavy-tailed and satisfy only weak moment assumptions?}}
\end{itemize}

The main contribution of this paper is to provide an affirmative answer to this question by showing that such recovery guarantees remain valid under remarkably weak assumptions. 
More precisely, we prove that both nuclear norm minimization~\eqref{nuclear} and semidefinite-constrained empirical risk minimization~\eqref{PSD} achieve uniform, stable, and robust recovery at the optimal sample complexity $m=\mathcal{O}\lk rn\rk$ up to moment-dependent constants, provided that the sampling vectors have independent, suitably normalized entries with only finite $4+\delta$ moments. 
To the best of our knowledge, this is the first result at this level of generality for the quadratic sampling model~\eqref{eq:measurement} in the heavy-tailed regime.
Moreover, our analysis yields two further byproducts.
First, we establish optimal sample complexity guarantees for low-rank matrix recovery under complex projective $4$-design sampling, improving earlier results by removing an extra logarithmic factor~\cite{kueng2017low,kabanava2016stable}.
Second, our arguments also yield stability guarantees for phaseless operators arising in phase retrieval under similarly weak moment assumptions.
Finally, numerical experiments corroborate our theoretical findings and demonstrate the effectiveness of these convex recovery procedures in heavy-tailed sampling settings.

Our approach differs in several essential ways from existing analyses. 
Its main novelty lies in two technical ingredients tailored to the heavy-tailed setting, which allow us to remove the Gaussian or sub-Gaussian assumptions while retaining the rank null space property (rank NSP) framework~\cite{kabanava2016stable} combined with Mendelson's small ball method~\cite{koltchinskii2015bounding,tropp2015convex}.
The first ingredient is a decoupling-based moment estimate for quadratic forms, which replaces the standard arguments commonly used under sub-Gaussian assumptions~\cite{chen2015exact,krahmer2020complex}.
The second ingredient draws on recent advances in covariance estimation for heavy-tailed distributions~\cite{tikhomirov2018sample,abdalla2024covariance,jirak2025concentration}, which we adapt to control the empirical process terms arising in the small ball method.
Together, these ingredients yield uniform, stable, and robust recovery guarantees under weak moment assumptions in the heavy-tailed setting.
We believe that both ingredients may be of independent interest.

We introduce some notation that will be used throughout the paper.
We denote by $\mathbb{S}_{\mathbb{C}}^{n-1}$ and $\mathbb{S}^{n-1}$ the unit spheres in $\mathbb{C}^n$ and $\mathbb{R}^n$, respectively.
For a matrix $\M$, we write $\lV\M\rV_*$, $\lV\M\rV_F$, and $\lV\M\rV_{op}$ for its nuclear norm, Frobenius norm, and operator norm, respectively.
For an integer $r\ge 1$, $\M^r$ denotes the best rank-$r$ approximation of $\M$, and $\M^{r,c}:=\M-\M^r$ denotes the residual part.
If $\x\in\mathbb{C}^n$, then $\Re\lk\x\rk$ and $\Im\lk\x\rk$ denote its real and imaginary parts, respectively.
For a random variable $X$ and $p\ge 1$, we write
$\lV X\rV_{L_p}:=\lk \E\lv X\rv^p\rk^{1/p}$.
Finally, for two nonnegative real sequences $\{a_t\}_t$ and $\{b_t\}_t$, we write $b_t=\mathcal{O}(a_t)$ (or $b_t\lesssim a_t$) if there exists a constant $C>0$ such that $b_t\le Ca_t$, and write $b_t\gtrsim a_t$ if there exists a constant $c>0$ such that $b_t\ge ca_t$. 
Similarly, we write $a\lesssim_p b$ if there exists a constant $C_p>0$, depending only on $p$, such that $a\le C_pb$, and write $a\gtrsim_p b$ if there exists a constant $c_p>0$, depending only on $p$, such that $a\ge c_pb$.

The remainder of the paper is organized as follows. 
In Section~\ref{sec:main}, we present the main recovery guarantees for the two convex programs. 
Section~\ref{subsec:pre} reviews preliminaries on the rank NSP and Mendelson's small ball method. 
In Sections~\ref{subsec:quad} and~\ref{subsec:covariance}, we establish moment estimates for quadratic forms and carry out the covariance estimation analysis.
Sections~\ref{subsec:proof1} and~\ref{subsec:proof2} are devoted to the proofs of our main results. 
In Section~\ref{sec:tdesign}, we present recovery results for complex projective $4$-design sampling. 
Section~\ref{sec:stability} is concerned with stability results for phase retrieval. 
Finally, Section~\ref{sec:experiments} presents numerical experiments that corroborate our theoretical findings.

\section{Main Results}\label{sec:main}

In this section, we present the main results of the paper.
We begin with the nuclear norm minimization program~\eqref{nuclear} for the quadratic sampling model~\eqref{eq:sample}. 
The following theorem shows that uniform, stable, and robust recovery of low-rank Hermitian matrices remains possible in the heavy-tailed setting under merely $4+\delta$ moment assumptions on the sampling vectors.

\begin{theorem}\label{thm:heavy1}
Let $\delta>0$ and $q\ge 1$.
Consider the noisy measurement process in~\eqref{eq:sample} where $\lV\pmb{\omega}\rV_{\ell_q} \le \eta$, with $m$ sampling matrices of the form $\left\{\va_k\va_k^*\right\}_{k=1}^m$.
Assume that $\left\{\va_k\right\}_{k=1}^m$ are independent copies of a random vector 
$\va \in \C^n$ with independent mean-zero, variance-one entries $\{a_i\}_{i=1}^n$ satisfying
\begin{equation}\label{eq:condition}
\alpha_{4+\delta}=\max_i\E\lz\lv a_i\rv^{4+\delta}\rz<\infty,\quad \beta=\min_i\E\lz\lv a_i\rv^4\rz>1,\quad\gamma=\max_i\lv\E\lz a_i^2\rz\rv<1.
\end{equation}
Set $\zeta=\min\{\beta-1,1-\gamma^2\}>0$.
Fix $r \le n$, and suppose that
\begin{equation*}
m \ge C_1\lk\delta\rk \cdot f \cdot rn.
\end{equation*}
Then, with probability at least $1-e^{-C_2 m \cdot g^2}$ it holds that for all $\M_0 \in \mathcal{H}_n$, any solution $\M^\sharp$ to the program~\eqref{nuclear} obeys
\begin{equation*}
\lV \M_0 - \M^\sharp\rV_F \le 
\frac{C_3}{\sqrt{r}} \lV\M_0^{r,c}\rV_{*} 
+ C_4\lk\delta\rk\frac{\eta}{h \cdot m^{1/q}}.
\end{equation*}
Here, $C_1\lk\delta\rk,C_4\lk\delta\rk$ are positive constants depending only on $\delta$, whereas $C_2$ and $C_3$ are positive universal constants.
Moreover, $f,g,h$ are constants given as
\begin{equation}\label{eq:fgh}
f=\frac{\alpha_{4+\delta}^{\frac{32+12\delta}{\lk4+\delta\rk\delta}}}{\zeta^{3+\frac{8}{\delta}}},
\quad g=\frac{\zeta^{4/\delta+1}}{\alpha_{4+\delta}^{4/\delta}},
\quad\text{and}\quad h=g\cdot\zeta^{1/2}.
 \end{equation} 
\end{theorem}

We next consider the semidefinite-constrained empirical risk minimization program~\eqref{PSD}, which is noise-blind. 
The following theorem shows that it enjoys a similar recovery guarantee in the heavy-tailed quadratic sampling setting.

\begin{theorem}\label{thm:heavy2}
Let $\delta>0$ and $q\ge 1$.
Consider the noisy measurement process in~\eqref{eq:sample} with $m$ sampling matrices of the form $\left\{\va_k\va_k^*\right\}_{k=1}^m$.
Here $\left\{\va_k\right\}_{k=1}^m$ are independent copies of a random vector 
$\va \in \C^n$ with independent mean-zero, variance-one entries $\{a_i\}_{i=1}^n$.
Assume that $\{a_i\}_{i=1}^n$ satisfy~\eqref{eq:condition}, and let $\zeta$ be defined as in Theorem~\ref{thm:heavy1}.
Fix $r \le n$, and suppose that
\begin{equation*}
m \ge C_1\lk\delta\rk \cdot f \cdot rn.
\end{equation*}
Then, with probability at least $1-e^{-C_2 m\cdot g^2}-e^{-2n}-\frac{1}{10m^{\delta/4}}-\frac{C_3\lk\delta\rk}{m}$ it holds that for all $\M_0 \succeq \pmb{0}$, any solution $\M^\sharp$ to the program~\eqref{PSD} obeys
\begin{equation*}
\lV \M_0 - \M^\sharp\rV_F \le 
\frac{C_4}{\sqrt{r}} \lV\M_0^{r,c}\rV_{*} 
+ C_5\lk\delta\rk\frac{\lV\pmb{\omega}\rV_{\ell_q}}{h\cdot m^{1/q}}.
\end{equation*}
Here, $C_1\lk\delta\rk,C_3\lk\delta\rk,C_5\lk\delta\rk$ are positive constants depending on $\delta$, whereas $C_2$ and $C_4$ are positive universal constants.
The constants $f,g,h$ are those defined in~\eqref{eq:fgh} and depend only on $\alpha_{4+\delta}$, $\zeta$, and $\delta$.
\end{theorem}

We make several remarks on the above two theorems.

\begin{remark}[Nearly Minimal Assumptions on $\alpha_{4+\delta}$]
In the above theorems, we do not require the entries of $\va$ to be i.i.d.\ Gaussian or sub-Gaussian, as is commonly assumed in the literature.
Instead, independence and finite $(4+\delta)$-th moments suffice.
This moment condition is nearly minimal within our analytical framework, since the small ball analysis~\cite{chen2015exact,krahmer2020complex,huang2025low} naturally involves fourth moments of the entries through second-moment estimates of $\va^*\M\va$.
The additional $\delta>0$ provides the extra integrability needed for uniform high probability control in the heavy-tailed setting.
Thus, under these near-minimal assumptions, our recovery guarantees still achieve the order-optimal sample complexity $m=\mathcal{O}\lk rn\rk$, up to constants depending on the moment parameters $\alpha_{4+\delta},\zeta$ and $\delta$.
However, we do not know whether the sample complexity and the recovery bounds are optimal with respect to these constants.
\end{remark}

\begin{remark}[Necessity of $\beta$ and $\gamma$]
The conditions on $\beta$ and $\gamma$ are needed to avoid certain ambiguities; see, e.g.,~\cite{krahmer2020complex}.
To illustrate this point, suppose in addition that the entries are i.i.d.\ copies of a random variable $a\in\C$. 
Then conditions~\eqref{eq:condition} on $\beta$ and $\gamma$ reduce to
\begin{equation}\label{eq:condition0}
\beta=\E \lz\lv a\rv^4\rz > 1, \qquad
\gamma=\lv\E \lz a^2\rz\rv < 1.
\end{equation}
If $\E\lz\lv a\rv^4\rz=\E\lz\lv a\rv^2\rz=1$, which includes the Bernoulli case $a\in\{\pm1\}$ with probability $1/2$, then $\lv a\rv=1$ almost surely.
Consequently, the rank-one matrices
$\left\{\pmb{e}_j\pmb{e}_j^*\right\}_{j=1}^n$ are indistinguishable, where  $\left\{\pmb{e}_j\right\}_{j=1}^n$ denote the standard basis vectors of $\C^n$. 
Similarly, if $\lv\E \lz a^2\rz\rv=\E\lz\lv a\rv^2\rz=1$, for instance when $a=\lambda \widetilde{a}$ for some fixed $\lambda\in\C$ with $\lv\lambda\rv=1$ and some real-valued random variable $\widetilde{a}$, then the rank-one matrix $\x_0\x_0^*$ cannot be distinguished from  $\overline{\x}_0\overline{\x}_0^*$, although in general $\x_0\x_0^*\neq\overline{\x}_0\overline{\x}_0^*$. 
\end{remark}

\begin{remark}[Uniform, Stable, and Robust Guarantees]
 Within the rank-NSP-based framework of~\cite{kabanava2016stable}, our theorems show that uniform (simultaneously for all admissible target matrices), stable (with respect to approximate low-rankness), and robust (with respect to measurement noise) recovery of low-rank Hermitian matrices remains possible in the heavy-tailed setting.
 Theorems~\ref{thm:heavy1} and~\ref{thm:heavy2} apply beyond the exactly low-rank setting, with the approximation error measured by $\lV\M_0^{r,c}\rV_*/\sqrt r$.
When $\M_0$ has rank at most $r$, this term vanishes, and hence any solution $\M^\sharp$ to~\eqref{nuclear} (or~\eqref{PSD}) satisfies
\begin{equation*}
\lV \M_0 - \M^\sharp\rV_F
\lesssim_{\alpha_{4+\delta},\zeta,\delta}
\frac{\eta}{m^{1/q}}
\quad
\lk\text{or}
\quad
\frac{\lV\pmb{\omega}\rV_{\ell_q}}{m^{1/q}}\rk.
\end{equation*}
Moreover, if $\M_0$ has rank at most $r$ and $\pmb{\omega}=\pmb{0}$ in the sampling model~\eqref{eq:sample}, then exact reconstruction holds.
\end{remark}

\begin{remark}[Phase Retrieval via PhaseLift]\label{re:PhaseLift}
If $\M_0$ has rank one, say $\M_0=\x_0\x_0^*$, then the noiseless part of the measurements takes the form
\begin{equation*}
y_k=\va_k^*\M_0\va_k = \lv\lg \va_k,\x_0\rg\rv^2,\quad k=1,\ldots,m.
\end{equation*}
Thus, the quadratic sampling model reduces to the phase retrieval problem~\cite{balan2009painless}.
In this setting, programs~\eqref{nuclear} and~\eqref{PSD} correspond to the well-known PhaseLift method~\cite{candes2013phaselift,candes2014solving}.
Consequently, when $m =\mathcal{O}\lk n\rk$, our results yield uniform and robust recovery guarantees for phase retrieval in the heavy-tailed setting.
\end{remark}

\begin{remark}[Probability Bounds]
The recovery guarantee in Theorem~\ref{thm:heavy1} holds with an exponentially high probability of the form
$1-e^{-\Omega\lk m\rk}$.
In contrast, the probability bound in Theorem~\ref{thm:heavy2} contains additional error terms, namely the exponential term $e^{-2n}$ and two
polynomially decaying terms, $\frac{1}{m^{\delta/4}}$ and $\frac{1}{m}$.
This loss is caused by the extra condition required in the proof of Theorem~\ref{thm:heavy2}, namely that $\sum_{k=1}^m \va_k\va_k^*$ be positive definite; see Section~\ref{subsec:NSP}.
\end{remark}

\begin{remark}[A Special Case: Eight-Moment Condition]
In the setting of Theorems~\ref{thm:heavy1} and~\ref{thm:heavy2}, suppose in addition that the entries of $\va$ have finite eighth moments, and set $\alpha:=\max_{1\le i\le n}\E\,\lv a_i\rv^8<\infty$.
Then, by taking $\delta=4$, the conclusions of Theorems~\ref{thm:heavy1} and~\ref{thm:heavy2} simplify as follows.
Assume that $m\gtrsim\frac{\alpha^{5/2}}{\zeta^5}rn$.
Then with probability at least $1-e^{-\Omega\lk \zeta^4m/\alpha^2\rk}$ it holds that for all $\M_0 \in \mathcal{H}_n$, any solution $\M^\sharp$ to the program~\eqref{nuclear} obeys
\begin{equation*}
\lV \M_0 - \M^\sharp\rV_F \lesssim 
\frac{\lV\M_0^{r,c}\rV_{*}}{\sqrt{r}}  
+ \frac{\alpha\eta}{\zeta^{5/2}m^{1/q}}.
\end{equation*}
Moreover, with probability at least $1-e^{-\Omega\lk \zeta^4m/\alpha^2\rk}-2e^{-n}-\mathcal{O}\!\lk1/m\rk$ it holds that for all $\M_0 \succeq \pmb{0}$, any solution $\M^\sharp$ to the program~\eqref{PSD} obeys
\begin{equation*}
\lV \M_0 - \M^\sharp\rV_F \lesssim
\frac{\lV\M_0^{r,c}\rV_{*} }{\sqrt{r}}
+ \frac{\alpha\lV\pmb{\omega}\rV_{\ell_q}}{\zeta^{5/2}m^{1/q}}.
\end{equation*}
\end{remark}

\section{Proof of Main Results}

\subsection{Preliminaries}\label{subsec:pre}

\subsubsection{Rank Null Space Property}\label{subsec:NSP}
We first recall some standard results on low-rank matrix recovery from~\cite{kabanava2016stable},
beginning with the Frobenius-robust rank null space property (rank NSP).

    \begin{definition}\label{def:NSP}
	For $q\geq 1$, we say an operator $\mathcal{A}\colon \mathcal{H}_n\to \R^m$ satisfies the Frobenius-robust rank NSP with respect to $\ell_q$ of order $r$ with constants $0<\rho<1$ and $\tau>0$ if for all $\M\in\mathcal{H}_n$, 
	\begin{equation*}
		\lV \M^r\rV_F \le \frac{\rho}{\sqrt{r}}\lV \M^{r,c}\rV_{*} +\tau\lV \mathcal{A}\lk \M\rk\rV_{\ell_q}.
	\end{equation*}
	\end{definition}

The recovery guarantee for the model~\eqref{nuclear} follows from the following proposition, provided that the quadratic sampling operator $\mathcal{A}$ satisfies the Frobenius-robust rank NSP.

\begin{proposition}[{\cite[Theorem 3.1]{kabanava2016stable}}]\label{lm:NSP0}
Let $\mathcal{A} \colon \mathcal{H}_n\to \R^m$ satisfy the Frobenius-robust rank NSP with respect to $\ell_q$ of order $r$ with constants $0 < \rho < 1$ and $\tau > 0$. 
Then for all $\M_0 \in \mathcal{H}_n$, any solution $\M^\sharp$ to~\eqref{nuclear} with
$\y = \mathcal{A}\lk \M_0\rk + \pmb{\omega}$ and $ \lV\pmb{\omega}\rV_{\ell_q} \le \eta$,
approximates $\M_0$ with error
\begin{equation*}
\lV \M_0 - \M^\sharp\rV_F \le 
\frac{C_1}{\sqrt{r}} \lV\M_0^{r,c}\rV_{*} 
+ D_1\tau\,\eta,
\end{equation*}
where $C_1=\frac{2\lk 1+\rho\rk^2}{\lk1-\rho\rk}$ and $D_1=\frac{2\lk 3+\rho\rk}{1-\rho}$.
\end{proposition}

Although the feasible set in the model~\eqref{PSD} is the positive semidefinite cone and does not explicitly impose a low-rank constraint, the model can still recover low-rank matrices under suitable structural conditions on the sampling matrices. 
Specifically, to prove Theorem~\ref{thm:heavy2}, in addition to requiring that $\mathcal{A}$ satisfies the Frobenius-robust rank NSP, we also need the empirical matrix $\pmb{W}:=\frac{1}{m}\sum_{k=1}^m\va_k\va_k^{*}$ to be positive definite. 

\begin{proposition}[{\cite[Theorem 8.1]{kabanava2016stable}}]\label{lm:NSP+PSD}
Suppose $\pmb{W}=\frac{1}{m}\sum_{k=1}^m\va_k\va_k^{*}$ is positive definite and $\mathcal{A} \colon \mathcal{H}_n\to \R^m$ satisfies the Frobenius-robust rank NSP with respect to $\ell_q$ of order $r$ with constants $0 < \rho < 1/\kappa\lk\pmb{W}\rk$ and $\tau > 0$, where $\kappa\lk\pmb{W}\rk:=\lV\pmb{W}\rV_{op}\cdot\lV\pmb{W}^{-1}\rV_{op}$. 
Then for all $\M_0 \succeq \pmb{0}$, any solution $\M^\sharp$ to~\eqref{PSD} with
$\y = \mathcal{A}\lk \M_0\rk + \pmb{\omega}$,
approximates $\M_0$ with error
\begin{equation*}
\lV \M_0 - \M^\sharp\rV_F \le 
\frac{C_2\kappa\lk\pmb{W}\rk}{\sqrt{r}} \lV\M_0^{r,c}\rV_{*} 
+ \lk \frac{C_2}{m^{1/q}\sqrt{r}}+D_2 \lV\pmb{W}\rV_{op}\tau\rk\,\lV\pmb{W}^{-1}\rV_{op}\,\lV\pmb{\omega}\rV_{\ell_q},
\end{equation*}
where $C_2=\frac{2\lk1+\rho\kappa\lk\pmb{W}\rk\rk^2}{1-\rho\kappa\lk\pmb{W}\rk}$ and $D_2=\frac{3+\rho\kappa\lk\pmb{W}\rk}{1-\rho\kappa\lk\pmb{W}\rk}$.
\end{proposition}

    The Frobenius-robust rank NSP for the quadratic sampling operator $\mathcal{A}$ plays a crucial role in proving both Theorem~\ref{thm:heavy1} and Theorem~\ref{thm:heavy2}.
	To this end, we introduce the set $\mathcal{T}_{\rho,r}$:
	\begin{equation*}
	\mathcal{T}_{\rho,r}=\left\{ \M\in \mathcal{H}_n\, :\,\lV \M\rV_F=1,\, \lV \M^r\rV_F>\frac{\rho}{\sqrt{r}}\lV \M^{r,c}\rV_{*} \right\}.
	\end{equation*}

The next lemma shows that establishing a uniform lower bound on $\lV\mathcal{A}\lk \M\rk\rV_{\ell_q}$ for all $\M\in \mathcal{T}_{\rho,r}$ is sufficient to guarantee that $\mathcal{A}$ satisfies the Frobenius-robust rank NSP.

\begin{lemma}[{\cite[Lemma~3.3]{kabanava2016stable}}]\label{lm:NSP}
If
\begin{equation}\label{eq:lower}
\inf \left\{ \lV\mathcal{A}\lk \M\rk\rV_{\ell_q} : \M \in \mathcal{T}_{\rho,r} \right\} > \frac{1}{\tau},
\end{equation}
then $\mathcal{A}$ satisfies the Frobenius-robust rank NSP with respect to $\ell_q$ of order $r$ with constants $\rho$ and $\tau$.
\end{lemma}

The following lemma characterizes the approximate low-rank structure of the set $\mathcal{T}_{\rho,r}$. 
It follows directly from Lemma~3.4 in~\cite{kabanava2016stable}, which embeds $\mathcal{T}_{\rho,r}$ into a scaled version of the convex hull of matrices of rank at most $r$ and unit Frobenius norm.

\begin{lemma}[{\cite[Lemma~3.4]{kabanava2016stable}}]\label{lm:diag}
For any $\M \in \mathcal{T}_{\rho,r}$, we have
\begin{equation}
\lV\M\rV_*\le\sqrt{1+\lk1+\rho^{-1}\rk^{2}}\sqrt{r}.
\end{equation}
\end{lemma}

\subsubsection{Small Ball Method}\label{subsec:SB}

    The standard approach to establishing the Frobenius-robust rank NSP for the quadratic sampling operator $\mathcal{A}$ through Lemma~\ref{lm:NSP} is Mendelson's small ball method~\cite{mendelson2015learning,koltchinskii2015bounding}, which has been extensively applied over the past decade to a variety of problems, including sparse recovery, low-rank matrix recovery, phase retrieval, and blind deconvolution~\cite{lecue2017sparse,abdalla2022dictionary,kueng2017low,kabanava2016stable,rauhut2019low,huang2025low,krahmer2020complex,krahmer2021convex}.
        
	\begin{proposition}[{\cite[Theorem 1.5]{koltchinskii2015bounding}, \cite[Proposition 5.1]{tropp2015convex}}]
 \label{lm:small_ball}
	Fix $\mT\subset\R^{n}$.
	Let $\{\pmb{\phi}_k\}_{k=1}^{m}$ be independent copies of a random vector $\pmb{\phi}$ in $\R^{n}$, and let $\{\varepsilon_{k}\}_{k=1}^{m}$ be a Rademacher sequence independent of $\{\pmb{\phi}_k\}_{k=1}^{m}$. 
    Define the small ball function and the supremum of the empirical process by
    \begin{equation*}	     
     \mathcal{Q}_{\xi}\lk \mT;\pmb{\phi}\rk=\inf_{\xu\in\mT}\mathbb{P}\lk\lv \lg\pmb{\phi},\xu\rg\rv\ge\xi\rk\quad\text{and}\quad
     \mathcal{W}_{m}\lk \mT;\pmb{\phi}\rk=\E\sup_{\xu\in\mT}\lv\frac{1}{m}\sum_{k=1}^{m}\varepsilon_{k}\lg\pmb{\phi}_k,\xu\rg\rv.
	         \end{equation*}
	Then for any $q\ge1,\xi>0$ and $t>0$, with probability at least $1-\exp\lk-2t^{2}\rk$, it holds that
	  \begin{equation}
\inf_{\xu\in\mT}\lk\sum_{k=1}^{m}\lv\lg\pmb{\phi}_k,\xu\rg\rv^{q}\rk^{1/q} \ge
	m^{\frac{1}{q}}\lk\xi\mathcal{Q}_{2\xi}\lk \mT;\pmb{\phi}\rk-2\mathcal{W}_{m}\lk \mT;\pmb{\phi}\rk-\frac{\xi t}{\sqrt{m}}\rk.
	    \end{equation}     
\end{proposition}

	Hence, to apply the preceding proposition to lower bound $\lV\mathcal{A}\lk \M\rk\rV_{\ell_q}$ in Lemma~\ref{lm:NSP} over the set $\mT_{\rho,r}$, it remains to establish a lower bound for
	\begin{equation*}
		\mathcal{Q}_{2\xi}:=\inf_{\M\in \mT_{\rho,r}}\mathbb{P}\lk\vert \va^{*}\M\va\vert\ge2\xi\rk 
	\end{equation*}
    and an upper bound for
    \begin{equation*}
	\mathcal{W}_m:=\E\sup_{\M\in \mT_{\rho,r}}\lv\left\lg\pmb{H},\M\right\rg\rv,
	\end{equation*}
	where $\pmb{H}=\frac{1}{m}\sum_{k=1}^m \varepsilon_k \va_k \va_k^{*}$.

To lower bound the small ball function $\mathcal{Q}_{2\xi}$, we employ the celebrated Paley--Zygmund inequality. 
In order to impose weaker assumptions on the sampling vector $\va$, we use the following generalized version of the Paley--Zygmund inequality~\cite{petrov2007lower}.
Its proof is postponed to Appendix~\ref{ap:PZ}.
\begin{fact}\label{fact:PZ}
Let $Z\ge 0$ be a random variable and let $0<s<q<\infty$ with $0<\lV Z\rV_{L_q}<\infty$. 
Then, for any $\theta\in(0,1)$,
\begin{equation}\label{eq:PZ0}
    \mathbb{P}\lk Z\ge \theta \lV Z\rV_{L_s}\rk
    \ge
    \lk 1-\theta^s\rk^{\frac{q}{q-s}}
    \lk\frac{\lV Z\rV_{L_s}}{\lV Z\rV_{L_q}}\rk^{\frac{sq}{q-s}} .
\end{equation}
\end{fact}
\noindent Applying Fact~\ref{fact:PZ} to
$Z=\lv \va^{*}\M\va\rv^{2}$ with $s=1$ and $q=p/2$, where $p>2$, yields that for any $t\in(0,1)$,
    \begin{equation}\label{eq:PZ}
	\mathbb{P}\lk\vert \va^{*}\M\va\vert^2\ge t\, \E\,\vert \va^{*}\M\va\vert^2\rk
    \ge \lk1-t\rk^{\frac{p}{p-2}}\cdot\frac{\lk\E\,\vert \va^{*}\M\va\vert^2\rk^{\frac{p}{p-2}}}{\lk\E\,\vert \va^{*}\M\va\vert^p\rk^{{\frac{2}{p-2}}}}.
	\end{equation} 
    
The following lemma from~\cite{krahmer2020complex} is particularly useful.
It not only explains the origin of the degeneracy in quadratic sampling, but can also be used to provide a lower bound for $\E\,\lv\va^*\M\va\rv^2$.
\begin{lemma}[{\cite[Lemma 9]{krahmer2020complex}}]\label{lm:E^2}
Let $\va\in\C^n$ be a random vector with independent mean-zero, variance-one entries $\{a_i\}_{i=1}^n$.
Then for any matrix $\M\in \mathcal{H}_n$ it holds that
\begin{align*}
\E\,\lv\va^*\M\va\rv^2
&= \lk\operatorname{Tr} \M\rk^2
+ \sum_{i=1}^n\lk\E\lz\lv a_i\rv^4\rz-1\rk \M_{ii}^2\\
&\quad +\sum_{i\ne j}\M_{ij}^2\,\E\lz \overline{a_i}^2\rz\cdot \E\lz a_j^2\rz+ \sum_{i\ne j}\lv\M_{ij}\rv^2.\\
\end{align*}
\end{lemma}
Hence, in order to lower bound $\mathcal{Q}_{2\xi}$, it remains to obtain an upper bound for $\E\,\vert \va^{*}\M\va\vert^p$ for some $p>2$.
A standard tool for this purpose is the Hanson--Wright inequality~\cite{rudelson2013hanson}; see, for instance,~\cite{chen2015exact,krahmer2020complex}.
However, this approach relies crucially on the sub-Gaussian assumption on $\va$.
In Section~\ref{subsec:quad}, we instead develop an alternative argument based on decoupling, which yields explicit moment bounds for the quadratic form $\lv\va^{*}\M\va\rv$ under suitable finite-moment assumptions.

For the supremum term $\mathcal{W}_m$, by Hölder's inequality and Lemma~\ref{lm:diag} in Section~\ref{subsec:NSP}, we obtain
\begin{equation}\label{eq:Wm}
\begin{aligned}
\mathcal{W}_m
&\le\E\sup_{\M\in \mT_{\rho,r}}\lV\M\rV_*\cdot\lV\pmb{H}\rV_{op}\\
&\le\sqrt{1+\lk1+\rho^{-1}\rk^{2}}\sqrt{r} \cdot\E\lV\pmb{H}\rV_{op},
\end{aligned}
\end{equation}
	where $\pmb{H}=\frac{1}{m}\sum_{k=1}^m \varepsilon_k \va_k \va_k^{*}$.
    Therefore, to upper bound $\mathcal{W}_m$, it suffices to control $\E\lV\pmb{H}\rV_{op}$.
A standard approach is based on covering number arguments (see, e.g.,~\cite{tropp2015convex,krahmer2020complex}), 
which typically relies on sub-Gaussian concentration and therefore breaks down in the heavy-tailed setting.
In Section~\ref{subsec:covariance}, we show how to upper bound $\E\lV\pmb{H}\rV_{op}$ using recent results on covariance matrix estimation, which allows us to establish recovery guarantees with optimal sample complexity.

\subsection{Quadratic Form}\label{subsec:quad}

We derive an upper bound for the moments of the quadratic form $\lv\va^{*}\M\va\rv$, which will serve as a key technical ingredient in our analysis.
The proof is based on a decoupling argument, in the spirit of the proof of the Hanson--Wright inequality~\cite{rudelson2013hanson}, 
but it does not rely on a sub-Gaussian assumption.
 
\begin{proposition}\label{thm:4order}
Let $\va\in\C^n$ be a random vector with independent mean-zero, variance-one entries $\{a_i\}_{i=1}^n$.
Fix $p\ge2$, and assume that
$\alpha_{2p}=\max_i\E\lz\lv a_i\rv^{2p}\rz<\infty$.
Then for any $\M\in\C^{n\times n}$,
    \begin{equation}\label{eq:4order}
            \E\,\lv\va^{*}\M\va\rv^{p} \le C_p \lk \lv\operatorname{Tr}\lk\M\rk\rv^p+\alpha_{2p}\lV\M\rV_F^{p}\rk,
            \end{equation}
            where $C_p$ is a constant depending only on $p$.
\end{proposition}

\begin{proof}[Proof of Proposition~\ref{thm:4order}]
The proof proceeds in five steps.

\textbf{Step 1. Expansion.} We write
\begin{equation*}
\begin{aligned}
\va^*\M\va
&=\sum_{i=1}^n \M_{ii}\lv a_i\rv^2+\sum_{i\neq j}\overline{a_i}\,\M_{ij}\,a_j\\
&=\operatorname{Tr}\lk\M\rk+\sum_{i=1}^n \M_{ii}\lk\lv a_i\rv^2-1\rk+\sum_{i\neq j}\overline{a_i}\,\M_{ij}\,a_j\\
&:= \operatorname{Tr}\lk \M\rk+D+S.
\end{aligned}
\end{equation*}
Therefore, by the elementary inequality
$\lv x+y+z\rv^p\lesssim_p \lv x\rv^p+\lv y\rv^p+\lv z\rv^p$,
it follows that
\begin{equation}\label{eq:san}
\E\lv\va^*\M\va\rv^p \lesssim_p  \lv\operatorname{Tr}\lk\M\rk\rv^p+\E\lv D\rv^p+\E\lv S\rv^p.
\end{equation}

\textbf{Step 2. Estimate $D$.} 
We begin with the diagonal part
\begin{equation*}
D=\sum_{i=1}^n X_i,\quad \text{where } X_i=\M_{ii}\lk\lv a_i\rv^2-1\rk.
\end{equation*}
Since the random variables $\{X_i\}_{i=1}^n$ are independent and mean-zero, we appeal to the following form of Rosenthal's inequality (see~\cite[Theorem~3]{rosenthal1970subspaces}):
for any independent mean-zero random variables $\{Y_i\}_{i=1}^n$ and any $p\ge2$,
\begin{equation}\label{eq:rosen}
\E\lv\sum_{i=1}^n Y_i\rv^p
\lesssim_p\lk\sum_{i=1}^n \E \lv Y_i\rv^2\rk^{p/2}+\sum_{i=1}^n \E\lv Y_i\rv^p.
\end{equation}
We next estimate the moments of $X_i$. 
For each $p\ge 2$, we write $\alpha_p:=\max_{i}\E\lz\lv a_i\rv^p\rz$.
Since $\E\lv a_i\rv^2=1$, we have
\begin{equation}\label{eq:E2}
\E\lv X_i\rv^2
=\lv \M_{ii}\rv^2\,\E\lv\lv a_i\rv^2-1\rv^2
=\lv \M_{ii}\rv^2\lk \E\lv a_i\rv^4-1\rk
\le\alpha_4 \,\lv \M_{ii}\rv^2.
\end{equation}
Moreover,
\begin{equation}\label{eq:Ep}
\E\lv X_i\rv^p
=\lv \M_{ii}\rv^p\,\E\lv\lv a_i\rv^2-1\rv^p\lesssim_p \alpha_{2p}\,\lv \M_{ii}\rv^p.
\end{equation}
Applying~\eqref{eq:rosen} to $\{X_i\}_{i=1}^n$, we obtain
\begin{equation}\label{eq:Diag}
\begin{aligned}
\E\lv D\rv^p
&\lesssim_p \lk\sum_{i=1}^n \E\lv X_i\rv^2\rk^{p/2}+\sum_{i=1}^n\E\lv X_i\rv^p\\
&\lesssim_p\alpha_{4}^{p/2}\lk\sum_{i=1}^n \lv \M_{ii}\rv^2\rk^{p/2}+\alpha_{2p}\sum_{i=1}^n \lv \M_{ii}\rv^p\\
&\lesssim_p\alpha_{2p}\,\lV\M\rV_F^p.
\end{aligned}
\end{equation}
Here, we used that for $p\ge 2$,
$\sum_{i=1}^n \lv \M_{ii}\rv^p\le\lk\sum_{i=1}^n \lv \M_{ii}\rv^2\rk^{p/2}\le\lV\M\rV_F^p$,
and that $\alpha_{2p}\ge\alpha_{4}^{p/2}\ge 1$.

\textbf{Step 3. Decoupling.}
To estimate the off-diagonal part $S=\sum_{i\neq j}\overline{a_i}\,\M_{ij}\,a_j$, we first invoke a decoupling argument.
Let $\va'$ be an independent copy of $\va$. 
Then, by the decoupling inequality for the quadratic form (see~\cite[Theorem~6.1.1]{vershynin2018high}), we obtain
\begin{equation*}
\E\lv S\rv^p
\lesssim_p\E\lv\sum_{i\neq j}\overline{a_i}\,\M_{ij}\,a_j'\rv^p.
\end{equation*}
Set
\begin{equation*}
\widetilde{S}:=\sum_{i\neq j}\overline{a_i}\,\M_{ij}\,a_j'
=\sum_{i=1}^n \overline{a_i}\,c_i,
\qquad
\text{where}\,c_i:=\sum_{j\neq i} \M_{ij}a_j'.
\end{equation*}
Conditioning on $\va'$, and applying Rosenthal's inequality~\eqref{eq:rosen}, we obtain
\begin{equation*}
\begin{aligned}
\E_{\va}\lv \widetilde{S}\rv^p
&\lesssim_p\lk\sum_{i=1}^n \lv c_i\rv^2\rk^{p/2}+\sum_{i=1}^n \lv c_i\rv^p\,\E\lv a_i\rv^p\\
&\lesssim_p\lk\sum_{i=1}^n \lv c_i\rv^2\rk^{p/2}+\alpha_p\lk\sum_{i=1}^n \lv c_i\rv^2\rk^{p/2}\\
&\lesssim_p\alpha_p \lk\sum_{i=1}^n \lv \sum_{j\neq i} \M_{ij}a_j'\rv^2\rk^{p/2}=\alpha_p\lV\M_{\mathrm{off}}\va'\rV_{\ell_2}^p.
\end{aligned}
\end{equation*}
Here, $\M_{\mathrm{off}}$ denotes the off-diagonal part of $\M$.
Taking expectation with respect to $\va'$, we conclude that
\begin{equation}\label{eq:Moff}
\E\lv S\rv^p\lesssim_p\E\lv \widetilde{S} \rv^p
\lesssim_p\alpha_{p}\, \E\lV\M_{\mathrm{off}}\va'\rV_{\ell_2}^p.
\end{equation}

\textbf{Step 4. Estimate $S$.}
To estimate $S$, we upper bound the term $\E\lV\M_{\mathrm{off}}\va'\rV_{\ell_2}^p$.
Let $\M_{\mathrm{off},j}$ denote the $j$-th column of $\M_{\mathrm{off}}$, and let $\{\varepsilon_j\}_{j=1}^n$ be a Rademacher sequence independent of $\va'$. 
By symmetrization (see Exercise~6.4.5 in~\cite{vershynin2018high}) and the Khintchine--Kahane inequality in $\mathbb{C}^n$ (see~\cite{kahane1985some}), we have
\begin{equation}\label{eq:Moff0}
\begin{aligned}
\E\lV\M_{\mathrm{off}}\va'\rV_{\ell_2}^p
&=\E_{\va'}\lV\sum_{j=1}^n a_j'\M_{\mathrm{off},j}\rV_{\ell_2}^p\\
&\le 2^p\E_{\va'}\E_{\pmb{\varepsilon}}\lV\sum_{j=1}^n \varepsilon_j a_j'\M_{\mathrm{off},j}\rV_{\ell_2}^p\\
&\lesssim_p\E_{\va'}\lk\E_{\pmb{\varepsilon}}\lV\sum_{j=1}^n \varepsilon_j a_j'\M_{\mathrm{off},j}\rV_{\ell_2}^2\rk^{p/2} \\
&=\E_{\va'}\lk\sum_{j=1}^n \lv a_j'\rv^2\lV\M_{\mathrm{off},j}\rV_{\ell_2}^2\rk^{p/2}.
\end{aligned}
\end{equation}
Since $p/2\ge1$, the triangle inequality yields
\begin{equation}\label{eq:Lq}
\lV\sum_{j=1}^n \lv a_j'\rv^2\lV\M_{\mathrm{off},j}\rV_{\ell_2}^2\rV_{L_{p/2}}
\le \sum_{j=1}^n \lV\M_{\mathrm{off},j}\rV_{\ell_2}^2\,\lV\lv a_j'\rv^2\rV_{L_{p/2}}
\le \alpha_{p}^{2/p}\lV\M_{\mathrm{off}}\rV_F^2.
\end{equation}
Therefore, combining~\eqref{eq:Moff0} and~\eqref{eq:Lq}, we obtain
\begin{equation*}
\E\lV\M_{\mathrm{off}}\va'\rV_{\ell_2}^p\lesssim_p\alpha_{p}\lV\M_{\mathrm{off}}\rV_F^p\le \alpha_{p}\lV\M\rV_F^p.
\end{equation*}
Combining this with~\eqref{eq:Moff}, we conclude that
\begin{equation}\label{eq:offdiag}
\E\lv S\rv^p\lesssim_p\alpha^2_{p}\,\lV\M\rV_F^p.
\end{equation}

\textbf{Step 5. Summary.}
Combining the estimates~\eqref{eq:san}, \eqref{eq:Diag}, and~\eqref{eq:offdiag}, we finally obtain
\begin{equation*}
\begin{aligned}
\E\lv\va^*\M\va\rv^p
&\lesssim_p\lv\operatorname{Tr}\lk\M\rk\rv^p+\lk\alpha_{2p}+\alpha_p^2\rk\lV\M\rV_F^p\\
&\lesssim_p\lv\operatorname{Tr}\lk\M\rk\rv^p+\alpha_{2p}\lV\M\rV_F^p.
\end{aligned}
\end{equation*}
Here, in the last step we used that $\alpha_{2p}\ge \alpha^2_{p}$.
\end{proof}

\subsection{Covariance Estimation}\label{subsec:covariance}

To upper bound $ \E\lV\frac{1}{m}\sum_{k=1}^m\varepsilon_k\va_k\va_{k}^{*}\rV_{op}$ 
in the heavy-tailed setting, we leverage recent advances in covariance matrix estimation for heavy-tailed distributions~\cite{tikhomirov2018sample,abdalla2024covariance}. 
The proof of the following theorem is inspired by the strategy developed in~\cite[Section 3.2]{jirak2025concentration},  
whereas our setting is complex-valued and requires some additional modifications.

\begin{theorem}\label{thm:chaos}
Let $\va\in\C^n$ be a random vector with independent mean-zero, variance-one entries $\{a_i\}_{i=1}^n$. 
Assume that $\alpha_p=\max_{i}\E\lz\lv a_i\rv^p\rz<\infty$ for some $p>4$.
Let $\left\{\va_k\right\}_{k=1}^m$ be independent copies of $\va$, and let $\{\varepsilon_{k}\}_{k=1}^{m}$ be a Rademacher sequence independent of $\left\{\va_k\right\}_{k=1}^m$. 
If $m\ge C_1 n$ for a sufficiently large constant $C_1 > 0$,
then
    \begin{equation}\label{eq:order4}
            \E\lV\frac{1}{m}\sum_{k=1}^m\varepsilon_k\va_k\va_{k}^{*}\rV_{op}
            \le C_2\lk p\rk\lk\alpha_p^{2/p}\sqrt{\frac{n}{m}}+\frac{n}{m}\rk,
            \end{equation}
            where $C_2\lk p\rk>0$ is a constant depending only on $p$.
\end{theorem}

The proof of the above theorem relies on the following two lemmas.
The first one is a non-asymptotic bound for covariance matrix estimation under heavy-tailed distributions. 

\begin{lemma}[\cite{tikhomirov2018sample,jirak2025concentration}]\label{lm:covariance}
Let $\xu\in\mathbb{R}^{n}$ be an isotropic random vector, and assume that $\sup_{\x\in\mathbb{S}^{n-1}}\lk\E\lv\lg\xu,\x\rg\rv^p\rk^{1/p}\le \kappa_p$ where $p>4$.
Let $\left\{\xu_k\right\}_{k=1}^m$ be independent copies of $\xu$. 
Assume that $m\ge C_3 n$ for a sufficiently large
constant $C_3 > 0$.
Then, with probability at least $1-e^{-n}-\frac{c\lk p\rk}{m}$, it holds that
\begin{equation}\label{eq:con of heavy}
\lV\frac{1}{m}\sum_{k=1}^m\xu_k\xu_{k}^{\top}-\pmb{I}_n\rV_{op}
\le
C_4\lk p\rk\,\lk \kappa_p^2\cdot\sqrt{\frac{n}{m}}+\frac{\max_{k}\lV\xu_k\rV_{\ell_2}^2}{m}\rk,
\end{equation}
where $c\lk p\rk$ and $C_4\lk p\rk$ are constants depending only on $p$.
\end{lemma}

\begin{remark}
The isotropic covariance estimate in Lemma~\ref{lm:covariance} goes back to Tikhomirov~\cite{tikhomirov2018sample}. 
The high-probability formulation stated above follows from equation~(36) in~\cite[Theorem~6]{jirak2025concentration}, specialized to the isotropic case. 
More general covariance estimates for heavy-tailed random vectors with nonidentity covariance were developed in~\cite{abdalla2024covariance,jirak2025concentration}, where the error bounds are expressed in terms of the \textit{effective rank} of the covariance matrix $\pmb{\Sigma}$, defined as
$r\lk\pmb{\Sigma}\rk=\operatorname{Tr}\lk\pmb{\Sigma}\rk/\lV\pmb{\Sigma}\rV_{op}$.
In the present paper, the isotropic version is sufficient for our purposes.
\end{remark}

To derive the desired expectation bound, we also require a Rosenthal-type inequality for heavy-tailed random matrices.

\begin{lemma}[{\cite[Theorem 3]{jirak2025concentration}}]\label{lm:heavymatrix}
Let $\X_1,\ldots,\X_m\in\mathbb{C}^{n\times n}$ be a sequence of centered, independent, Hermitian random matrices.
Set $M=\max_{k}\lV \X_k\rV_{op}$, and $\sigma^2=\lV\pmb{V}^2_m\rV_{op}$, where $\pmb{V}^2_m\succeq\sum_{k=1}^m \E\,\X_k^2$.
Then for all $p\ge1$, we have the following moment inequality:
\begin{equation*}
\lk\E\lV\sum_{k=1}^m \X_k\rV_{op}^p\rk^{1/p}
\le C_5\lk\sqrt{Q}\sigma+Q\,\E M+\frac{p}{\log(ep)}\,\lk\E M^p\rk^{1/p}\rk,
\end{equation*}
where $C_5>0$ is an absolute constant and $Q:=\max\left\{\log\lk r\lk\pmb{V}^2_m \rk\rk ,p\right\}$.
\end{lemma}

\subsubsection{Proof of Theorem~\ref{thm:chaos}}\label{subsec:proofchaos}

The proof proceeds in five steps.
In \textbf{Step 1}, we extend Lemma~\ref{lm:covariance} to complex-valued isotropic random vectors and obtain a high probability estimate. 
In \textbf{Steps 2--4}, we apply Lemma~\ref{lm:heavymatrix} to convert it into an expectation bound. 
In \textbf{Step 5}, we use the assumptions on $\va$ in Theorem~\ref{thm:chaos} to eliminate the maximum term.

\textbf{Step 1: From $\R^n$ to $\C^n$.}
We first use Lemma~\ref{lm:covariance} to prove the following lemma.
\begin{lemma}\label{lm:chaos}
Let $\va\in\mathbb{C}^{n}$ be an isotropic random vector such that $\sup_{\x\in\mathbb{S}_\mathbb{C}^{n-1}} \lk\E\lv\lg\va,\x\rg\rv^p\rk^{1/p}\le \tilde{\kappa}_p$ where $p>4$.
Let $\left\{\va_k\right\}_{k=1}^m$ be independent copies of $\va$. 
Assume that $m\ge \widetilde{C} n$ for a sufficiently large
constant $\widetilde{C} > 0$.
Then with probability at least $1-e^{-2n}-\frac{\tilde{c}\lk p\rk}{m}$, it holds that
\begin{equation}
\lV\frac{1}{m}\sum_{k=1}^m\va_k\va_{k}^{*}-\pmb{I}_n\rV_{op}
\le
\widetilde{C}\lk p\rk\,\lk\tilde{\kappa}_p^2\cdot\sqrt{\frac{n}{m}}+\frac{\max_{k}\lV\va_k\rV_{\ell_2}^2}{m}\rk.
\end{equation}
Here,  $\tilde{c}\lk p\rk$ and $\widetilde{C}\lk p\rk$ denote constants depending only on $p$.
\end{lemma}

\begin{proof}[Proof of Lemma~\ref{lm:chaos}]
Let $\theta$ be uniformly distributed on $[0,2\pi)$ and independent of $\va$, and define $\vb=e^{\mathrm{i}\theta}\va$.
Then $\vb\vb^*=\va\va^*$ and $\lV\vb\rV_{\ell_2}=\lV\va\rV_{\ell_2}$.
Thus, it suffices to prove the desired bound with $\va$ replaced by $\vb$.
Define the realifications of $\vb$ by
\begin{equation*}
\widetilde{\vb}:=
\begin{pmatrix}
\Re\lk\vb\rk\\
\Im\lk\vb\rk
\end{pmatrix}
\in\mathbb{R}^{2n},
\qquad
\widehat{\vb}:=
\begin{pmatrix}
-\Im\lk\vb\rk\\
\Re\lk\vb\rk
\end{pmatrix}
\in\mathbb{R}^{2n}.
\end{equation*}
Next, for any $\M=\M_1+\mathrm{i}\M_2\in\C^{n\times n}$, define its realification by
\begin{equation*}
\mathcal{R}\lk \M\rk
=
\begin{pmatrix}
\M_1 & -\M_2\\
\M_2 & \M_1
\end{pmatrix}\in\mathbb R^{2n\times 2n}.
\end{equation*}
It is well known that the realification preserves the operator norm, namely
$\lV\mathcal{R}\lk \M\rk\rV_{op}=\lV\M\rV_{op}$.

By the random phase construction, $\widetilde{\vb}\in\mathbb{R}^{2n}$ satisfies (see Appendix~\ref{ap:realification})
\begin{equation}\label{eq:isotropic}
\E\,\widetilde{\vb}\widetilde{\vb}^{\top}
=\frac{1}{2}\pmb{I}_{2n},
\qquad
\sup_{\x\in\mathbb{S}^{2n-1}}
\lk\E\lv\left\langle \widetilde{\vb},\x\right\rangle\rv^p\rk^{1/p}
\le \tilde{\kappa}_p.
\end{equation}
Let $\{\theta_k\}_{k=1}^m$ be independent copies of $\theta$, independent of $\{\va_k\}_{k=1}^m$, and set
$\vb_k=e^{\mathrm{i}\theta_k}\va_k$.
We apply the covariance estimate in Lemma~\ref{lm:covariance} to the isotropic random vector $\sqrt{2}\widetilde{\vb}$ in dimension $2n$. 
Therefore, provided $m\ge \widetilde{C}n$,
it follows that with probability at least $1-e^{-2n}-\frac{c\lk p\rk}{m}$, the following inequalities hold simultaneously:
\begin{equation*}
\left\{
\begin{aligned}
\lV
\frac{1}{m}\sum_{k=1}^m \widetilde{\vb}_k\widetilde{\vb}_k^\top-\frac{1}{2}\pmb{I}_{2n}
\rV_{op}
&\lesssim_p
\tilde{\kappa}_p^2\cdot\sqrt{\frac{2n}{m}}+\frac{\max_k\lV\widetilde{\vb}_k\rV_{\ell_2}^2}{m},\\[6pt]
\lV\frac{1}{m}\sum_{k=1}^m \widehat{\vb}_k\widehat{\vb}_k^\top-\frac{1}{2}\pmb{I}_{2n}\rV_{op}
&\lesssim_p
\tilde{\kappa}_p^2\cdot\sqrt{\frac{2n}{m}}+\frac{\max_k\lV\widehat{\vb}_k\rV_{\ell_2}^2}{m}.
\end{aligned}
\right.
\end{equation*}
Indeed, the second inequality follows from the first one since $\widehat{\vb}_k$ is obtained from $\widetilde{\vb}_k$ by an orthogonal transformation.

Moreover, a direct computation shows that (see Appendix~\ref{ap:Raa})
\begin{equation}\label{eq:Raa}
\mathcal{R}\lk \vb\vb^*\rk
=\widetilde{\vb}\widetilde{\vb}^{\top}+\widehat{\vb}\widehat{\vb}^{\top}.
\end{equation}
Consequently,
\begin{equation*}
\mathcal{R}\lk
\frac{1}{m}\sum_{k=1}^m \vb_k\vb_k^*-\pmb{I}_n\rk
=
\lk\frac{1}{m}\sum_{k=1}^m \widetilde{\vb}_k\widetilde{\vb}_k^\top-\frac{1}{2}\pmb{I}_{2n}\rk
+
\lk\frac{1}{m}\sum_{k=1}^m \widehat{\vb}_k\widehat{\vb}_k^\top-\frac{1}{2}\pmb{I}_{2n}\rk.
\end{equation*}
Using the identity $\lV\mathcal{R}\lk \M\rk\rV_{op}=\lV\M\rV_{op}$ together with the triangle inequality and the fact that $\lV\widetilde{\vb}_k\rV_{\ell_2}=\lV\widehat{\vb}_k\rV_{\ell_2}=\lV\vb_k\rV_{\ell_2}$, we obtain
\begin{equation*}
\lV\frac{1}{m}\sum_{k=1}^m \vb_k\vb_k^*-\pmb{I}_n\rV_{op}
\lesssim_p
\tilde{\kappa}_p^2\cdot\sqrt{\frac{n}{m}}+\frac{\max_k\lV\vb_k\rV_{\ell_2}^2}{m}.
\end{equation*}
Finally, since $\vb_k\vb_k^*=\va_k\va_k^*$ and $\lV\vb_k\rV_{\ell_2}=\lV\va_k\rV_{\ell_2}$ for all $k$, the desired estimate follows.
\end{proof}

\textbf{Step 2: Reduction to Expectation.}
Under the moment assumptions on the entries of $\va$, we have the following estimate
(see Appendix~\ref{ap:mu8}):
     \begin{equation}\label{eq:mu8}
      \sup_{\x\in\mathbb{S}_{\C}^{n-1}}\E \lv\lg\va,\x\rg\rv^p\lesssim_p \alpha_p.
    \end{equation}
    Let $\mathcal{E}$ denote the event, with probability at least $1-e^{-2n}-\frac{\tilde{c}\lk p\rk}{m}$, on which the inequality in Lemma~\ref{lm:chaos} holds.
    Therefore,
\begin{equation}\label{eq:splitbound}
\begin{aligned}
\E&\lV\frac{1}{m}\sum_{k=1}^{m} \va_k \va_k^{*} - \pmb{I}_n\rV_{op}\\
&\lesssim_p \alpha^{2/p}_p\sqrt{\frac{n}{m}}+
\frac{\E\max_k \lV\va_k\rV_{\ell_2}^{2}}{m} +
\E\lz\lV\frac{1}{m}\sum_{k=1}^{m} \va_k \va_k^{*} - \pmb{I}_n
\rV_{op}\mathbf{1}_{\mathcal{E}^{c}}\rz.
\end{aligned}
\end{equation}
By H\"older's inequality,
\begin{equation}\label{eq:holder}
\begin{aligned}
\E\left[\lV\frac{1}{m}\sum_{k=1}^{m} \va_k \va_k^{*} - \pmb{I}_n\rV_{op}
\mathbf{1}_{\mathcal{E}^{c}}\right]
&\le\lk\E\lV\frac{1}{m}\sum_{k=1}^{m} \va_k \va_k^{*} - \pmb{I}_n\rV_{op}^{2}\rk^{1/2}
\cdot\mathbb{P}\lk\mathcal{E}^c\rk^{1/2}\\
&\le\lk\E\lV\frac{1}{m}\sum_{k=1}^{m} \va_k \va_k^{*} - \pmb{I}_n\rV_{op}^{2}\rk^{1/2}
\cdot \lk e^{-2n}+\frac{\tilde{c}\lk p\rk}{m}\rk^{1/2}.
\end{aligned}
\end{equation}

\textbf{Step 3: Rosenthal's Inequality.} 
We now invoke Lemma~\ref{lm:heavymatrix}.
We set 
\begin{equation*}
\X_k=\frac{1}{m}\lk\va_k\va_k^{*}-\pmb{I}_n\rk, 
\qquad 
\pmb{V}_m^{2}=\sum_{k=1}^m \E\,\X_k^2.
\end{equation*}
Then
\begin{equation*}
M=\frac{\max_{k}\lV\va_k\va_k^{*}-\pmb{I}_n\rV_{op}}{m}\le\frac{\max_{k}\lV\va_k\rV_{\ell_2}^2+1}{m}.
\end{equation*}
By a direct calculation (see Appendix~\ref{ap:diag}), 
\begin{equation}\label{eq:diag}
\E\lV\va\rV_{\ell_2}^2\va\va^{*}=\lk n-1\rk\pmb{I}_n +\operatorname{diag}\lk \E\lz \lv a_1\rv^4\rz,\ldots,\E\lz \lv a_n\rv^4\rz \rk.
\end{equation}
Consequently,
\begin{equation*}
\sigma^2=\lV\sum_{k=1}^m \E\,\X_k^2\rV_{op}=\frac{1}{m^2}\lk\lV\sum_{k=1}^m\lk\E\lV\va_k\rV_{\ell_2}^2\va_k\va_k^{*}-\pmb{I}_n\rk\rV_{op}\rk\le \frac{\alpha_4+n-2}{m}\le \frac{\alpha_4 n}{m}.
\end{equation*}
We apply Lemma~\ref{lm:heavymatrix} with moment parameter $2$ and use $Q\lesssim\log\lk en\rk$.
This yields
\begin{equation*}
\begin{aligned}
&\lk\E\lV\frac{1}{m}\sum_{k=1}^{m} \va_k \va_k^{*} - \pmb{I}_n\rV_{op}^{2}\rk^{1/2}\\
&\qquad\qquad\lesssim 
\sqrt{\log \lk en\rk}\sqrt{\frac{\alpha_4 n}{m}}+
\frac{\log \lk en\rk}{m}
\E\max_k\lV\va_k\rV_{\ell_2}^{2}+
\frac{1}{m}\lk\E\max_k\lV\va_k\rV_{\ell_2}^{4}\rk^{1/2}.
\end{aligned}
\end{equation*}
Substituting the above inequality into~\eqref{eq:holder} and using $m\gtrsim\log^2 n$, we deduce that
\begin{equation}\label{eq: E bound}
\begin{aligned}
\E\lz\lV\frac{1}{m}\sum_{k=1}^{m} \va_k \va_k^{*} - \pmb{I}_n\rV_{op}
\mathbf 1_{\mathcal{E}^c}\rz\lesssim_p
\sqrt{\frac{\alpha_4 n}{m}}+
\frac{\E\max_k\lV\va_k\rV_{\ell_2}^{2}}{m}+
\frac{\lk\E\max_k\lV\va_k\rV_{\ell_2}^{4}\rk^{1/2}}{m\min\{m^{1/2},e^{n}\}}.
\end{aligned}
\end{equation}

\textbf{Step 4: Symmetrization Argument.}
 We now apply the symmetrization argument:
\begin{equation}\label{eq: E0 bound}
\begin{aligned}
\E\lV\frac{1}{m}\sum_{k=1}^{m} \varepsilon_k\va_k \va_k^{*} \rV_{op}
&\le
\E_{\va}\E_{\pmb{\varepsilon}}\lV\frac{1}{m}\sum_{k=1}^{m}\varepsilon_k\lk \va_k \va_k^{*} - \pmb{I}_n\rk\rV_{op}+\E_{\pmb{\varepsilon}}\lV\frac{1}{m}\sum_{k=1}^{m} \varepsilon_k\pmb{I}_n \rV_{op}\\
&\le
2\E_{\va}\lV\frac{1}{m}\sum_{k=1}^{m} \va_k \va_k^{*} - \pmb{I}_n\rV_{op}+\E_{\pmb{\varepsilon}}\lv\frac{1}{m}\sum_{k=1}^{m} \varepsilon_k\rv\\
&
\lesssim_p 
\lk \alpha_p^{2/p}+\alpha_4^{1/2}\rk\sqrt{\frac{ n}{m}}+
\frac{\E\max_k\lV\va_k\rV_{\ell_2}^{2}}{m}+
\frac{\lk\E\max_k\lV\va_k\rV_{\ell_2}^{4}\rk^{1/2}}{m\min\{m^{1/2},e^{n}\}}.
\end{aligned}
\end{equation}
The first line follows from the triangle inequality.
The second line follows from the symmetrization inequality; see Lemma~6.4.2 in~\cite{vershynin2018high}.
The final line follows from~\eqref{eq:splitbound} and~\eqref{eq: E bound}, together with the standard estimate
$\E\lv\sum_{k=1}^{m} \varepsilon_k\rv\lesssim\sqrt{m}$.

\textbf{Step 5: Bounding the Maximum Terms.}
We now bound the two maximal terms. 
By a direct calculation,
\begin{equation*}
\E\lV\va\rV_{\ell_2}^2=n,\quad\E\,\lV\va\rV_{\ell_2}^4\le\alpha_4 n+n\lk n-1\rk,\quad\text{and}\quad\operatorname{Var}\lk\lV\va\rV_{\ell_2}^2\rk\le\lk\alpha_4-1\rk n.
\end{equation*}
Therefore, 
\begin{equation*}
\begin{aligned}
\E\max_k\lV\va_k\rV_{\ell_2}^{2}&\le n+\E\max_k\lv\lV\va_k\rV_{\ell_2}^{2}-n\rv\le n+\E\lk\sum_{k=1}^m\lv\lV\va_k\rV_{\ell_2}^{2}-n\rv^2\rk^{1/2}\\
&\le n+\lk\E\sum_{k=1}^m\lv\lV\va_k\rV_{\ell_2}^{2}-n\rv^2\rk^{1/2}=n+\lk m \,\operatorname{Var}\lk\lV\va\rV_{\ell_2}^2\rk\rk^{1/2}\\
&\le n+\sqrt{\alpha_4 mn}.
\end{aligned}
\end{equation*}
Similarly,
 \begin{equation*}
\begin{aligned}
\lk\E\max_{1\le k\le m}\lV\va_k\rV_{\ell_2}^{4}\rk^{1/2}&=\lz\E\lk\max_{1\le k\le m}\lV\va_k\rV_{\ell_2}^{2}\rk^2\rz^{1/2}
\le n+\lk\E\max_{1\le k\le m}\lv\lV\va_k\rV_{\ell_2}^{2}-n\rv^2\rk^{1/2}\\
&\le n+\lk\E\sum_{k=1}^m\lv\lV\va_k\rV_{\ell_2}^{2}-n\rv^2\rk^{1/2}\le n+\sqrt{\alpha_4 mn}.
\end{aligned}
\end{equation*}
Combining the above two estimates with~\eqref{eq: E0 bound}, and using the fact that $\alpha_4^{1/4}\le\alpha_p^{1/p}$, finally we obtain
\begin{equation*}
\E\lV\frac{1}{m}\sum_{k=1}^{m} \varepsilon_k\va_k \va_k^{*} \rV_{op}
\lesssim_p
\alpha_p^{2/p}\sqrt{\frac{n}{m}}+
\frac{n}{m}.
\end{equation*}

\subsection{Proof of Theorem~\ref{thm:heavy1}}\label{subsec:proof1}

Now we are ready to prove Theorem~\ref{thm:heavy1}.
Recall that $\alpha_{4+\delta}=\max_i\E\lz\lv a_i\rv^{4+\delta}\rz$, $\beta=\min_i\E\lz\lv a_i\rv^4\rz>1$, $\gamma=\max_i\lv\E\lz a_i^2\rz\rv<1$ and define $\zeta=\min\{\beta-1,1-\gamma^2\}>0$.
By Lemma~\ref{lm:E^2} in Section~\ref{subsec:SB}, we have 
\begin{align*}
\E\,\lv\va^*\M\va\rv^2
&\ge \lk\operatorname{Tr} \M\rk^2
+ \lk\beta-1\rk\sum_{i=1}^n \M_{ii}^2+\lk1-\gamma^2\rk\sum_{i\ne j}\lv\M_{ij}\rv^2\\
&\ge\lk\operatorname{Tr} \M\rk^2+\zeta\cdot\lV\M\rV_F^2.
\end{align*}
Combining this estimate with Proposition~\ref{thm:4order} in Section~\ref{subsec:quad}
yields that for any nonzero $\M\in\mathcal{H}_n$
\begin{equation}\label{eq:2:4}
\frac{\lk\E\,\vert \va^{*}\M\va\vert^2\rk^{\frac{4+\delta}{\delta}}}{\lk\E\,\vert \va^{*}\M\va\vert^{2+\frac{\delta}{2}}\rk^{\frac{4}{\delta}}}\gtrsim_{\delta}\frac{\zeta^{\frac{4+\delta}{\delta}}}{\alpha_{4+\delta}^{4/\delta}}.
\end{equation}
Now by~\eqref{eq:PZ} in Section~\ref{subsec:SB}, we obtain 
\begin{equation}\label{eq:Q}
\begin{aligned}
\mathcal{Q}_{2\xi}
&:=\inf_{\M\in \mT_{\rho,r}}\mathbb{P}\lk\vert \va^{*}\M\va\vert\ge2\xi\rk\\
&\ge\inf_{\lV\M\rV_F=1}\mathbb{P}\lk\vert \va^{*}\M\va\vert^2\ge \frac{4\xi^2}{\zeta}\, \E\,\vert \va^{*}\M\va\vert^2\rk\\
&\ge\lk1-\frac{4\xi^2}{\zeta}\rk^{\frac{4+\delta}{\delta}}\inf_{\lV\M\rV_F=1}\frac{\lk\E\,\vert \va^{*}\M\va\vert^2\rk^{\frac{4+\delta}{\delta}}}{\lk\E\,\vert \va^{*}\M\va\vert^{2+\frac{\delta}{2}}\rk^{\frac{4}{\delta}}}\\
&\gtrsim_{\delta}\frac{\lk\zeta-4\xi^2\rk^{\frac{4+\delta}{\delta}}}{\alpha_{4+\delta}^{4/\delta}}.
\end{aligned}
\end{equation}
By~\eqref{eq:Wm} in Section~\ref{subsec:SB} and Theorem~\ref{thm:chaos} in Section~\ref{subsec:covariance},
provided that $m \gtrsim n$, we obtain
\begin{equation}\label{eq:W}
\mathcal{W}_m
\lesssim_{\delta} \sqrt{1+\lk1+\rho^{-1}\rk^{2}}\sqrt{r} \cdot\lk\alpha_{4+\delta}^{2/\lk4+\delta\rk}\sqrt{\frac{n}{m}}+\frac{n}{m}\rk.
\end{equation}

We choose $\xi=\frac{\zeta^{1/2}}{2\sqrt{2}}$ and $t=c_1\sqrt{m}\cdot g$ in Proposition~\ref{lm:small_ball} in Section~\ref{subsec:SB}.
Provided that
\begin{equation*}
m\gtrsim_{\delta} \rho^{-2}\cdot f \cdot rn,
\end{equation*} 
we obtain, with probability at least $1-e^{-c_2 m \cdot g^2}$
\begin{equation*}
\inf \left\{ \lV\mathcal{A}\lk \M\rk\rV_{\ell_q} : \M \in \mathcal{T}_{\rho,r} \right\} \gtrsim_{\delta} h \cdot m^{1/q},
 \end{equation*} 
where $c_1,c_2>0$ are sufficiently small absolute constants and $f,g,h$ are the constants defined in Theorem~\ref{thm:heavy1}.
Thus $\mathcal{A}$ satisfies the Frobenius-robust rank NSP in Lemma~\ref{lm:NSP} 
in Section~\ref{subsec:NSP} with constants $\rho$ and $\tau\lesssim_{\delta}\frac{1}{h \cdot m^{1/q}}$.
Finally, applying Proposition~\ref{lm:NSP0} in Section~\ref{subsec:NSP} and choosing $\rho=\frac{1}{2}$ completes the proof.

\subsection{Proof of Theorem~\ref{thm:heavy2}}\label{subsec:proof2}

In Section~\ref{subsec:proof1}, we established that $\mathcal{A}$ satisfies the Frobenius-robust rank NSP.
Therefore, by Proposition~\ref{lm:NSP+PSD} in Section~\ref{subsec:NSP}, in order to prove Theorem~\ref{thm:heavy2}, it remains to show that $\pmb{W}$ is positive definite and that its condition number is bounded.
To this end, we first establish the following fact.
We postpone the proof to Appendix~\ref{ap:max}.

\begin{fact}\label{lm:max}
Let $\va\in\C^n$ be a random vector with independent mean-zero, variance-one entries $\{a_i\}_{i=1}^n$. 
Assume that
$\alpha_p=\max_{i}\E\lz\lv a_i\rv^p\rz<\infty$ where $p\ge4$.
Let $\left\{\va_k\right\}_{k=1}^m$ be independent copies of $\va$.
Then with probability at least $1-\frac{1}{10m^{p/4-1}}$,
\begin{equation*}
\max_{k}\lV\va_k\rV_{\ell_2}^2\le  2n+C\lk p\rk\alpha_p^{2/p}\sqrt{mn}.
\end{equation*}
Here, $C\lk p\rk> 0$ is a sufficiently large constant depending only on $p$.
\end{fact}

Now, set $p=4+\delta$.
By Lemma~\ref{lm:chaos}, Equation~\eqref{eq:mu8} in Section~\ref{subsec:proofchaos}, and Fact~\ref{lm:max} above, we have, with probability at least $1-e^{-2n}-\frac{1}{10m^{\delta/4}}-\frac{\tilde{c}\lk \delta \rk}{m}$, that
\begin{equation}
1-C_1\lk \delta\rk \lk\alpha_{4+\delta}^{2/\lk4+\delta\rk}\sqrt{\frac{n}{m}}+\frac{n}{m}\rk\le \lambda_{\min}\lk\pmb{W}\rk\le \lambda_{\max}\lk\pmb{W}\rk
\le
1+C_1\lk \delta\rk \lk\alpha_{4+\delta}^{2/\lk4+\delta\rk}\sqrt{\frac{n}{m}}+\frac{n}{m}\rk,
\end{equation}
provided $m\gtrsim n$.
Consequently, if $m\gtrsim_\delta\alpha_{4+\delta}^{4/\lk4+\delta\rk} n$, then
\begin{equation*}
\max\left\{\lV\pmb{W}\rV_{op},\lV\pmb{W}^{-1}\rV_{op},\lV\pmb{W}\rV_{op}\cdot{\lV\pmb{W}^{-1}\rV_{op}}\right\}\le C_0.
\end{equation*}
Here, $C_0> 0$ denotes an absolute constant.
Finally, intersecting this event with the rank-NSP event from Section~\ref{subsec:proof1}, applying Proposition~\ref{lm:NSP+PSD}, and choosing $\rho=\frac{1}{2C_0}$ completes the proof.

\section{Complex Projective \texorpdfstring{$t$}{t}-Design}\label{sec:tdesign}

In this section, we consider sampling matrices formed by taking outer products of vectors drawn independently from a complex projective $t$-design~\cite{gross2015partial,kueng2017low,kabanava2016stable,gilles2025stable}. 
Such designs are finite sets of unit vectors in $\C^n$ that reproduce the low-order moment structure of Haar-random vectors, and therefore serve as a versatile tool for partially derandomizing recovery results. 
This is especially important in low-rank matrix recovery, and in particular in quantum state tomography~\cite{gross2010quantum}, where $t$-designs provide structured sampling ensembles that retain the theoretical advantages of fully random measurements while being closely connected to physically realizable implementations, such as random quantum circuits.

\begin{definition}\label{def:tdesign}
Let $\left\{\vw_1,\ldots,\vw_N\right\}\subseteq \C^n$ be a collection of unit vectors with corresponding weights
$\left\{p_1,\ldots,p_N\right\}\subseteq [0,1]$ satisfying $\sum_{i=1}^N p_i = 1$.
We say that the weighted set $\left\{p_i,\vw_i\right\}_{i=1}^N$ forms a weighted complex projective $t$-design if
\begin{equation}\label{eq:tdesign}
\sum_{i=1}^N p_i \lk \vw_i\vw_i^* \rk^{\otimes t}
=\int_{\mathbb{S}^{n-1}_{\C}} (\vw\vw^*)^{\otimes t}\,d\mu(\vw).
\end{equation}
Here, $\mu$ denotes the normalized uniform measure on $\mathbb{S}^{n-1}_{\C}$.
\end{definition}

We present the following theorem, which states that low-rank matrix recovery from complex projective $4$-design sampling can be achieved with the optimal sample complexity $m=\mathcal{O}\lk rn\rk$.
This removes the extra logarithmic factor in the sample complexity bounds of~\cite{kueng2017low,kabanava2016stable}, where $m=\mathcal{O}\lk rn\log n\rk$ is required.
 
\begin{theorem}\label{thm:4design}
Let $\{p_i,\vw_i\}_{i=1}^N$ be a weighted complex projective $4$-design, and define $\widetilde{\vw}_i=\sqrt[4]{n(n+1)}\,\vw_i$.
Consider the noisy measurement process in~\eqref{eq:sample} with $m$ sampling matrices $\left\{\va_k\va_k^*\right\}_{k=1}^m$, where $\left\{\va_k\right\}_{k=1}^m$ are independent copies of a random vector 
$\va \in \C^n$ drawn from $\{p_i,\widetilde{\vw}_i\}_{i=1}^N$.
Fix $r \le n$, and suppose that $m \ge C_1  rn$ for a sufficiently large constant $C_1>0$.
Then the following statements hold.
\begin{itemize}
        \item[(a)] If $\lV\pmb{\omega}\rV_{\ell_q} \le \eta$, then with probability at least $1-e^{-C_2 m}$, it holds that for all $\M_0 \in \mathcal{H}_n$, any solution $\M^\sharp$ to the program~\eqref{nuclear} obeys
\begin{equation*}
\lV \M_0 - \M^\sharp\rV_F \le 
\frac{C_3}{\sqrt{r}} \lV\M_0^{r,c}\rV_{*} 
+ C_4\frac{\eta}{m^{1/q}}.
\end{equation*}
    \item[(b)]
With probability at least $1-e^{-C_5 m}-e^{-2n}-\frac{C_6}{m}$, it holds that for all $\M_0 \succeq \pmb{0}$, any solution $\M^\sharp$ to the program~\eqref{PSD} obeys
\begin{equation*}
\lV \M_0 - \M^\sharp\rV_F \le 
\frac{C_7}{\sqrt{r}} \lV\M_0^{r,c}\rV_{*} 
+ C_8\frac{\lV\pmb{\omega}\rV_{\ell_q}}{m^{1/q}}.
\end{equation*}
\end{itemize}
Here, $C_1,\dots,C_8$ are positive universal constants.
\end{theorem} 

The following lemma is a key ingredient in the proof of the above theorem.

\begin{lemma}\label{lm:3design}
Let $\{p_i,\vw_i\}_{i=1}^N$ be a weighted complex projective $3$-design, and set
$\widetilde{\vw}_i=\sqrt[4]{n(n+1)}\,\vw_i$.
Let $\left\{\va_k\right\}_{k=1}^m$ be independent copies of a random vector 
$\va \in \C^n$ drawn from $\{p_i,\widetilde{\vw}_i\}_{i=1}^N$.
Let $\{\varepsilon_{k}\}_{k=1}^{m}$ be a Rademacher sequence independent of $\left\{\va_k\right\}_{k=1}^m$. 
If $m\ge C_9 n$ for a sufficiently large constant $C_9 > 0$,
then
    \begin{equation}\label{eq:3design}
            \E\lV\frac{1}{m}\sum_{k=1}^m\varepsilon_k\va_k\va_{k}^{*}\rV_{op}
            \le C_{10}\lk\sqrt{\frac{n}{m}}+\frac{n}{m}\rk,
            \end{equation}
           where $C_{10}$ is a positive universal constant.
\end{lemma}

\begin{proof}[Proof of Lemma~\ref{lm:3design}]
By the definition of a weighted complex projective $3$-design, for any 
$\x\in\mathbb{S}_\mathbb{C}^{n-1}$, we have
\begin{equation*}
\begin{aligned}
\E\lv\langle\va,\x\rangle\rv^6
&=\lk\lk n+1\rk n\rk^{3/2}\E\lv\langle\vw,\x\rangle\rv^6\\
&=\lk\lk n+1\rk n\rk^{3/2}
\int_{\mathbb{S}^{n-1}_{\C}} \lv\langle\vw,\x\rangle\rv^6\,d\mu(\vw)
\\
&=\frac{3! \lk\lk n+1\rk n\rk^{3/2}}{n(n+1)(n+2)}\le 6.
\end{aligned}
\end{equation*}
Similarly, for any 
$\x\in\mathbb{S}_\mathbb{C}^{n-1}$,
$\E\lv\langle\va,\x\rangle\rv^2
=\sqrt{\frac{n+1}{n}}\ge 1$.
Consequently, the complex projective $3$-design ensemble satisfies the $L_6$--$L_2$ condition
\begin{equation*}
\lk\E\lv\langle\va,\x\rangle\rv^6\rk^{1/6}
\le 6^{1/6}\lk\E\lv\langle\va,\x\rangle\rv^2\rk^{1/2},
\qquad \forall\,\x\in\C^{n}.
\end{equation*}
After normalization, we apply Lemma~\ref{lm:chaos} in Section~\ref{subsec:proofchaos} with $p=6$ (although $\va$ is not centered, the random phase construction used in Section~\ref{subsec:proofchaos} allows us to replace it by the centered vector $e^{\mathrm{i}\theta}\va$ without changing the sampling matrix $\va\va^*$). 
We obtain that if $m\gtrsim n$, then with probability at least $1-e^{-2n}-\frac{c_1}{m}$,
\begin{equation}\label{eq:tdesign0}
\lV\frac{1}{m}\sum_{k=1}^m\va_k\va_{k}^{*}-\sqrt{\frac{n+1}{n}}\pmb{I}_n\rV_{op}
\lesssim
\sqrt{\frac{n}{m}}+\frac{\max_{k}\lV\va_k\rV_{\ell_2}^2}{m}\lesssim \sqrt{\frac{n}{m}}+\frac{n}{m}.
\end{equation}
Here, we used the identities 
\begin{equation*}
\E\,\va\va^{*}
=\sqrt{\frac{n+1}{n}}\pmb{I}_n
\quad \text{and} \quad
\max_k \lV\va_k\rV_{\ell_2}^2
=\lV\va\rV_{\ell_2}^2=\sqrt{\lk n+1\rk n}.
\end{equation*}
By the same argument as in the proof of Theorem~\ref{thm:chaos} in Section~\ref{subsec:covariance}, using~\eqref{eq:tdesign0} together with Lemma~\ref{lm:heavymatrix} in Section~\ref{subsec:covariance}, we obtain the desired estimate~\eqref{eq:3design} whenever $m\gtrsim n$.
\end{proof}

We are now ready to prove Theorem~\ref{thm:4design}.
\begin{proof}[Proof of Theorem~\ref{thm:4design}]
We first prove part~(a) of Theorem~\ref{thm:4design}.
As in the proof of Theorem~\ref{thm:heavy1}, it suffices to establish a uniform lower bound on $\lV\mathcal{A}\lk\M\rk\rV_{\ell_q}$ over $\mathcal{T}_{\rho,r}$ by using Proposition~\ref{lm:small_ball} in Section~\ref{subsec:SB}; Lemma~\ref{lm:NSP} in Section~\ref{subsec:NSP} then yields the Frobenius-robust rank NSP.
By~\cite[Proposition~12]{kueng2017low}, for every $\xi \in [0, 1/2]$,
\begin{equation}\label{eq:Q4}
\mathcal{Q}_{2\xi}:=\inf_{\M \in \mathcal{T}_{\rho, r}} \mathbb{P} \lk\lv\lg\va\va^*,\M\rg\rv \ge 2\xi\rk \ge \inf_{\lV \M\rV_F=1} \mathbb{P} \lk\lv\lg\va\va^*,\M\rg\rv \ge 2\xi\rk \ge \frac{\lk 1-4\xi^2\rk^2}{24}.
\end{equation}
Moreover, by~\eqref{eq:Wm} in Section~\ref{subsec:SB} and Lemma~\ref{lm:3design}, provided that $m\gtrsim n$,
\begin{equation*}
\mathcal{W}_m:=\E\sup_{\M\in \mT_{\rho,r}}\lv\left\lg\pmb{H},\M\right\rg\rv
\lesssim \sqrt{1+\lk1+\rho^{-1}\rk^{2}} \cdot\sqrt{\frac{rn}{m}},
\end{equation*}
where $\pmb{H} = \frac{1}{m}\sum_{k=1}^m \varepsilon_k \va_k \va_k^*$. 
Finally, choosing $\xi=1/4, t=c_2\sqrt{m}$ in Proposition~\ref{lm:small_ball}, we obtain, provided that $m\gtrsim \rho^{-2}rn$, that $\mathcal{A}$ satisfies the Frobenius-robust rank NSP with constants $\rho$ and $\tau\lesssim m^{-1/q}$ with probability at least $1-e^{-c_3m}$.
Taking $\rho=\frac{1}{2}$ and applying Proposition~\ref{lm:NSP0} in Section~\ref{subsec:NSP} proves part~(a).

We next prove part~(b). 
Following the proof of Theorem~\ref{thm:heavy2}, it remains, by Proposition~\ref{lm:NSP+PSD} in Section~\ref{subsec:NSP}, to show that $\pmb{W}=\frac{1}{m}\sum_{k=1}^m\va_k\va_k^{*}$ is positive definite and has bounded condition number.
By~\eqref{eq:tdesign0}, provided that $m\gtrsim n$,  with probability at least $1-e^{-2n}-\frac{c_1}{m}$,
\begin{equation*}
\sqrt{\frac{n+1}{n}}-C_1 \sqrt{\frac{n}{m}}\le \lambda_{\min}\lk\pmb{W}\rk\le \lambda_{\max}\lk\pmb{W}\rk
\le
\sqrt{\frac{n+1}{n}}+C_1 \sqrt{\frac{n}{m}}.
\end{equation*}
Hence, if  $m\ge \widetilde{C} n$ for a sufficiently large constant $\widetilde{C}>0$, then $\pmb{W}$ is positive definite and
\begin{equation*}
\max\left\{\lV\pmb{W}\rV_{op},\lV\pmb{W}^{-1}\rV_{op},\lV\pmb{W}\rV_{op}\cdot{\lV\pmb{W}^{-1}\rV_{op}}\right\}\le C_0.
\end{equation*}
Intersecting this event with the rank-NSP event obtained above, choosing $\rho=\frac{1}{2C_0}$, and applying Proposition~\ref{lm:NSP+PSD} proves part~(b).
\end{proof}

\begin{remark}
Recovery guarantees for approximate complex projective $4$-designs have also been established in~\cite{kueng2017low,kabanava2016stable} under suitable assumptions, with sample complexity $m=\mathcal{O}\lk rn\log n\rk$; see~\cite[Section 3.1]{kueng2017low} for the definition of approximate complex projective $t$-designs. 
We expect that our method can be extended to this setting, yielding analogous guarantees with the optimal sample complexity $m=\mathcal{O}\lk rn\rk$. 
The main obstacle to extending our approach to this setting is that the ingredient in our proof, namely Lemma~\ref{lm:chaos}, are used in a form that relies on the exact moment structure of complex projective $4$-designs. 
One possible way to overcome this difficulty is to extend Lemma~\ref{lm:chaos} to complex random vectors with a general covariance matrix.
We leave this direction for future work.
\end{remark}  

\begin{remark}
The preceding argument can also be adapted to complex projective $3$-design ensembles. 
Recent work~\cite{gilles2025stable} established recovery guarantees in this setting with sample complexity $m=\mathcal{O}\lk r^3 n\log n\rk$. 
More precisely, by following the preceding argument with the small ball estimate~\eqref{eq:Q4} replaced by~\cite[Lemma~3.2.1]{gilles2025stable}, one can improve the sample complexity to $m=\mathcal{O}\lk r^3 n\rk$. 
We omit the details here.
\end{remark}

\section{Stability of Phase Retrieval}\label{sec:stability}

Phase retrieval refers to the problem of reconstructing an unknown signal $\x_{0}\in\C^{n}$ from $m$ phaseless measurements, 
given in the form of intensities $\left\{\lv\lg\va_{k},\x_0 \rg\rv^2\right\}_{k=1}^m$,
where $\Omega:=\left\{ \va_{k} \right\}_{k=1}^{m}\subseteq\C^n$ denotes the known collection of sampling vectors.
In practical applications, ensuring robust reconstruction performance is perhaps the most important concern.
Theorem~\ref{thm:heavy1} and Theorem~\ref{thm:heavy2} state that, under rather weak assumptions, phase retrieval can be achieved by robustly recovering the rank-one matrix $\x_0\x_0^*$ via the PhaseLift method~\cite{candes2013phaselift,candes2014solving}; see Remark~\ref{re:PhaseLift} in Section~\ref{sec:main}.
In this section, we approach the problem from a different perspective and focus on characterizing the stability of the following phaseless operator~\cite{eldar2014phase,balan2015invertibility}:

\begin{equation}\label{eq:phaseless}
{\mF}_{\Omega}: \C^{n}/\sim\rightarrow\R^{m},\quad\quad	{\mF}_{\Omega}\lk\x\rk=\lk\lv\lg\va_{1},\x\rg\rv^{2}, \ldots, \lv\lg\va_{m},\x\rg\rv^{2} \rk^{\top}.
	\end{equation}
Here, since reconstruction is inherently possible only up to a global phase, we consider the quotient space $\C^{n}/\sim$ as the input domain, 
where the equivalence relation $\sim$ is given by $\x\sim\y$ if $\x=e^{\mathrm{i}\theta}\y$ for some $\theta\in[0,2\pi)$.
Following~\cite{eldar2014phase,balan2015invertibility,balan2016reconstruction,duchi2019solving}, we give the definition of stability of the phaseless operator $\mF_{\Omega}$.

\begin{definition}\label{def:stable}
	The phaseless operator $\mF_{\Omega}:\C^{n}/\sim\rightarrow\R^{m}$ is said to be $C$-stable on $\C^n$ with respect to $\lV\,\cdot\,\rV_{\ell_q}$ for some constant $C>0$ if, for every $\x,\y\in\C^n$, one has
	\begin{equation}\label{eq:stable}
		\lV \mF_{\Omega}\lk\x\rk- \mF_{\Omega}\lk\y\rk \rV_{\ell_q} \ge C\cdot \textbf{dist}^2 \lk \x,\y \rk.
		\end{equation}		
		Here, the distance $\textbf{dist} \lk \x,\y \rk$ is defined as $ \textbf{dist}\lk\x, \y\rk := \min\limits_{\theta\in[0,2\pi)} \lV e^{\mathrm{i}\theta}\x - \y \rV_{\ell_2}$.
\end{definition}

The main result of this section is as follows.

\begin{theorem}\label{thm:PR}
Let $\delta>0$ and $q\ge 1$.
Consider the phaseless operator $\mF_{\Omega}$ with $m$ sampling vectors $\Omega=\left\{\va_k\right\}_{k=1}^m$.
Here, $\left\{\va_k\right\}_{k=1}^m$ are independent copies of a random vector 
$\va \in \C^n$ with independent mean-zero, variance-one entries $\{a_i\}_{i=1}^n$.
Assume that $\{a_i\}_{i=1}^n$ satisfy~\eqref{eq:condition}, and let $\zeta$ be defined as in Theorem~\ref{thm:heavy1}.
Suppose that $m \ge C_1\lk\delta\rk \cdot f \cdot n$.
Then, with probability at least $1-e^{-C_2 m\cdot g^2}$, it holds that for all $\x,\y\in\C^n$,
\begin{equation*}
\lV \mF_{\Omega}\lk\x\rk- \mF_{\Omega}\lk\y\rk \rV_{\ell_q} \ge C_3\lk\delta\rk\cdot h\cdot m^{1/q}\cdot \textbf{dist}^2 \lk \x,\y \rk.
\end{equation*}
Here, $C_1\lk\delta\rk,C_3\lk\delta\rk$ are constants depending on $\delta$ and $C_2$ denotes a positive universal constant.
The constants $f,g,h$ are those defined in Theorem~\ref{thm:heavy1} and depend only on $\alpha_{4+\delta}$, $\zeta$, and $\delta$.
\end{theorem}

\begin{remark}
The stability of the phaseless operator $\mF_{\Omega}$ in the real case, measured in the $\ell_1$-norm, was first studied in~\cite{eldar2014phase}, where the analysis relies on a sub-Gaussian assumption on the sampling vectors. 
Later,~\cite[Proposition~1]{duchi2019solving} extended this result by showing that only a small ball condition is required.
However, in the complex setting, their analysis still relies on a sub-Gaussian assumption; see~\cite[Proposition~2]{duchi2019solving}. 
In contrast, our result does not rely on any sub-Gaussian assumption.
We also mention that related stability results for the amplitude model
$\{\lv\lg \va_k,\x_0\rg\rv\}_{k=1}^m$ have been considered in ongoing work~\cite{abdalla_ramos_taylor_inprep}, where the assumptions are formulated in terms of a rank-$2$ small ball condition.
This is a different setting from that of the present theorem.
\end{remark}

\begin{remark}
The stability result in Theorem~\ref{thm:PR} also applies to the complex projective $t$-design setting discussed in Section~\ref{sec:tdesign}. 
Indeed, for phase retrieval, the lifted difference $\x\x^*-\y\y^*$ has rank at most $2$, so the proof only requires the corresponding small ball and empirical process estimates on rank-$2$ Hermitian matrices. 
Consequently, complex projective $4$-design measurements yield stability of $\mF_{\Omega}$ with sample complexity $m=\mathcal{O}\lk n\rk$.
\end{remark}

\begin{proof}[Proof of Theorem~\ref{thm:PR}]
    By~\cite[Lemma~A.4]{duchi2019solving} or \cite[Proposition~1]{huang2025stable}, we have $\textbf{dist}^2 \lk \x,\y \rk\le2\lV\x\x^*-\y\y^*\rV_F$, which implies that
    \begin{equation*}
    \begin{aligned}
\inf_{\substack{\x,\y\in\C^n\\\x\x^*\neq \y\y^*}}\frac{\lV \mF_{\Omega}\lk\x\rk- \mF_{\Omega}\lk\y\rk \rV_{\ell_q}}{\textbf{dist}^2 \lk \x,\y \rk} 
&\ge \inf_{\substack{\x,\y\in\C^n\\\x\x^*\neq \y\y^*}}\frac{\lV \mF_{\Omega}\lk\x\rk- \mF_{\Omega}\lk\y\rk \rV_{\ell_q}}{2\lV\x\x^*-\y\y^*\rV_F}\\
&=\inf_{\substack{\x,\y\in\C^n\\\x\x^*\neq \y\y^*}}\frac{1}{2}\lk\sum_{k=1}^m\lv\left\langle\va_k\va_k^*,\frac{\x\x^*-\y\y^*}{\lV\x\x^*-\y\y^*\rV_F}\right\rangle\rv^q\rk^{1/q}\\
&\ge\inf_{\substack{\M\in\mathcal{H}_n,\lV\M\rV_F=1\\ \operatorname{rank}\lk\M\rk\le2}}\frac{1}{2}\lk\sum_{k=1}^m\lv\left\langle\va_k\va_k^*,\M\right\rangle\rv^q\rk^{1/q}
\end{aligned}
\end{equation*}
Now set $\widetilde\mT:=\left\{\M\in\mathcal{H}_n:\lV\M\rV_F=1, \operatorname{rank}\lk\M\rk\le2\right\}$.
We apply the small ball method from Proposition~\ref{lm:small_ball} in Section~\ref{subsec:SB}.
First, by~\eqref{eq:Q} in Section~\ref{subsec:proof1}, we have
\begin{equation*}	     
     \mathcal{Q}_{2\xi}\lk \widetilde\mT;\va\va^*\rk=\inf_{\M\in\widetilde\mT}\mathbb{P}\lk\lv \lg\va\va^*,\M\rg\rv\ge2\xi\rk
     \gtrsim_{\delta}\frac{\lk\zeta-4\xi^2\rk^{\frac{4+\delta}{\delta}}}{\alpha_{4+\delta}^{4/\delta}}.
	         \end{equation*}
    Moreover, by Theorem~\ref{thm:chaos} in Section~\ref{subsec:covariance},
    \begin{equation*}	
    \begin{aligned}
       \mathcal{W}_{m}\lk \widetilde\mT;\va\va^*\rk
       &=\E\sup_{\M\in\widetilde\mT}\lv\frac{1}{m}\sum_{k=1}^{m}\varepsilon_{k}\left\langle\va_k\va_k^*,\M \right\rangle\rv\\
       &\le\sqrt{2}\,\E\,\lV\frac{1}{m}\sum_{k=1}^{m}\varepsilon_{k}\va_k\va_k^*\rV_{op}
       \lesssim_{\delta} \sqrt{2}\lk\alpha_{4+\delta}^{2/\lk4+\delta\rk}\sqrt{\frac{n}{m}}+\frac{n}{m}\rk.
       \end{aligned}
	         \end{equation*}
             Here we used  
$\lV\M\rV_*\le\sqrt{\operatorname{rank}\lk\M\rk}\cdot\lV\M\rV_F\le\sqrt{2}$.
Finally, choosing $\xi=\frac{\zeta^{1/2}}{2\sqrt{2}}$ and $t=c_1\sqrt{m}\cdot g$ as in the proof of Theorem~\ref{thm:heavy1} in Section~\ref{subsec:proof1}, Proposition~\ref{lm:small_ball} implies that, provided $m\ge C_1\lk\delta\rk f n$, with probability at least $1-e^{-C_2 m\cdot g^2}$,
             \begin{equation*}
    \begin{aligned}
\inf_{\substack{\x,\y\in\C^n\\\x\x^*\neq \y\y^*}}\frac{\lV \mF_{\Omega}\lk\x\rk- \mF_{\Omega}\lk\y\rk \rV_{\ell_q}}{\textbf{dist}^2 \lk \x,\y \rk} 
&\ge\inf_{\M\in \widetilde\mT}\frac{1}{2}\lk\sum_{k=1}^m\lv\left\langle\va_k\va_k^*,\M\right\rangle\rv^q\rk^{1/q}\gtrsim_{\delta} h\cdot m^{1/q}.
\end{aligned}
\end{equation*}

\end{proof}

\section{Numerical Experiments}\label{sec:experiments}

We conduct numerical experiments to validate our theoretical results. 
We compare a heavy-tailed sampling ensemble with the standard complex Gaussian ensemble.
In both cases, the sampling vectors $\left\{\va_{k}\right\}_{k=1}^m$ are independent copies of a random vector $\va=(a_1,\ldots,a_n)^{\top}$ with i.i.d.\ entries. 
For the heavy-tailed ensemble, each entry is an independent copy of $a=\sqrt{\frac{3}{10}}\lk X+\mathrm{i}Y\rk$, where 
$X,Y\stackrel{\mathrm{i.i.d.}}{\sim}t_5$. 
Since $\operatorname{Var}\lk t_5\rk=5/3$, this normalization gives 
$\mathbb{E}\lz\lv a\rv^2\rz=1$, while $a$ has finite absolute moment of order $q$ if and only if $q<5$. 
For the Gaussian ensemble, each entry is an independent copy of 
$a\sim\mathcal{CN}\lk0,1\rk$.
In all trials, the ground-truth matrix $\M_0\in \mathbb{C}^{n\times n}$ is randomly generated in the form $\M_0=\pmb{Z}\pmb{Z}^*$ with $\pmb{Z}\in\mathbb{C}^{n\times r}$, and then normalized so that $\lV\M_0\rV_F=1$.
For both recovery models~\eqref{nuclear} and~\eqref{PSD}, the data-fidelity term is measured in the $\ell_2$-norm.

\paragraph{Phase transition.}

We first examine the empirical sample complexity in the noiseless setting. 
We fix $n=50$ and $r=3$, and vary the number of measurements through the oversampling ratio
$m/(rn)\in\{3,3.5,4,4.5,5,5.5,6\}$.
For each value of $m/(rn)$, each sampling ensemble, and each recovery model, we run 20 independent trials. 
In every trial, both the sampling vectors and the ground-truth matrix are generated independently. 
A trial is declared successful if the Frobenius reconstruction error satisfies
   $ \lV\pmb{M}^{\sharp}-\pmb{M}_0\rV_F<5\cdot 10^{-3}$.
The empirical success probability is then computed as the fraction of successful trials among the 20 repetitions. 
Figure~\ref{fig:phase-transition} reports the resulting phase transition curves. 
The left panel corresponds to the nuclear norm minimization model~\eqref{nuclear} (NNM model), while the right panel corresponds to the semidefinite-constrained empirical risk minimization model~\eqref{PSD} (PSD model). 
This experiment confirms that, for both recovery models, the number of measurements required for successful recovery scales on the order of $rn$, and that the heavy-tailed Student-$t_5$ ensemble exhibits phase transition behavior essentially similar to that of the Gaussian benchmark.

\begin{figure}[htbp]
    \centering
    \includegraphics[width=1.0\linewidth]{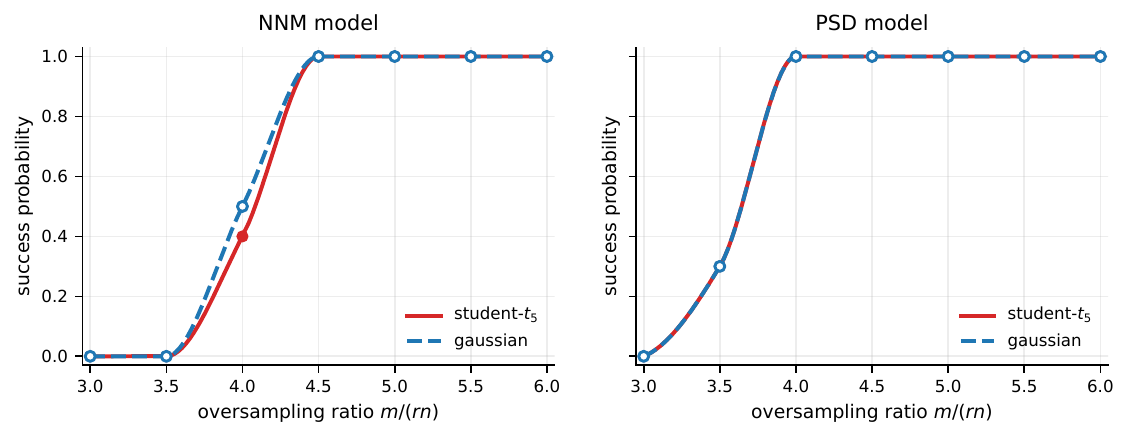}
     \caption{Phase transition for the Student-$t_5$ and Gaussian ensembles.}
    \label{fig:phase-transition}
\end{figure}

\paragraph{Noise robustness.}
We next investigate the robustness of the two recovery models under additive noise. 
We take $n=30$ and $r=3$, and fix the oversampling ratio at $m/(rn)=4.5$.
For each normalized noise level
\begin{equation*}
    \nu=\lV\pmb{\omega}\rV_{\ell_2}/\sqrt m
    \in\{10^{-4},3\cdot 10^{-4},10^{-3},3\cdot 10^{-3},10^{-2},3\cdot 10^{-2}\},
\end{equation*}
and for each sampling ensemble and each recovery model, we run 20 independent trials. 
In each trial, the noise vector is generated with independent real Gaussian entries and then rescaled so that
$\lV\pmb{\omega}\rV_{\ell_2}=\nu\sqrt m$.
We report the mean Frobenius reconstruction error $\lV\pmb{M}^{\sharp}-\pmb{M}_0\rV_F$ over the 20 trials. 
Since $\lV\pmb{M}_0\rV_F=1$, this error is also the relative Frobenius error. 
Figure~\ref{fig:noise-robustness} shows the mean reconstruction error as a function of the normalized noise level. 
Both the NNM model~\eqref{nuclear} and the PSD model~\eqref{PSD} are included in the comparison.
The dotted reference line has slope one. 
The curves are approximately parallel to this reference line, indicating that the reconstruction error grows nearly linearly with the noise level. 
Moreover, the Student-$t_5$ ensemble exhibits robustness comparable to the Gaussian benchmark, which is consistent with the robustness predicted by our results.

\begin{figure}[htbp]
    \centering
    \includegraphics[width=1.0\linewidth]{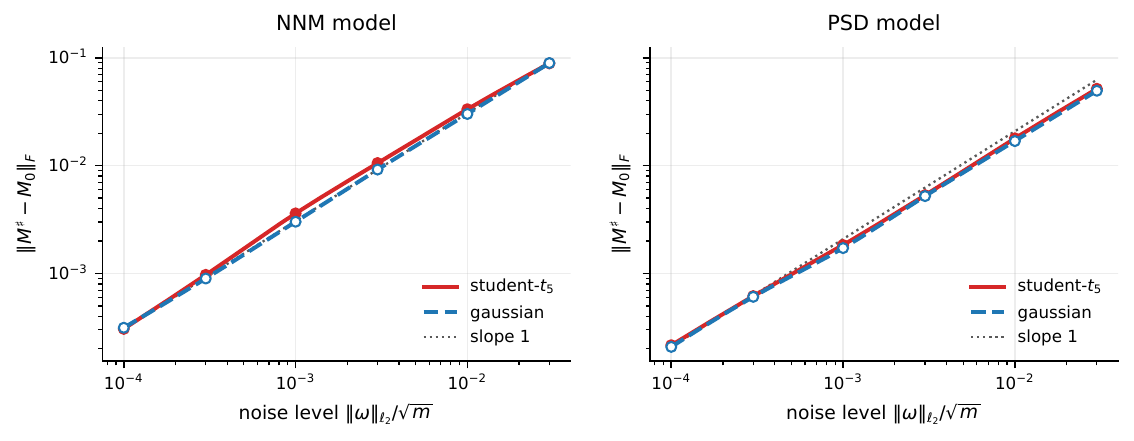}
    \caption{Noise robustness under the Student-$t_5$ and Gaussian ensembles.}
    \label{fig:noise-robustness}
\end{figure}

\appendix

\section{Proof of Fact~\ref{fact:PZ}}\label{ap:PZ}

For completeness, we include the proof; see also equation~(9) in~\cite{petrov2007lower}.
Let $Z,s,q,\theta$ be as in Fact~\ref{fact:PZ}, and define
$\mathcal{E}:=\{Z\ge \theta \lV Z\rV_{L_s}\}$. 
Since $Z^s<\theta^s\lV Z\rV_{L_s}^s$ on $\mathcal{E}^c$, we have
\begin{equation}\label{eq:PZ1}
    \lV Z\rV_{L_s}^s
    =\E\lk Z^s\mathbf{1}_{\mathcal{E}}\rk+ \E\lk Z^s\mathbf{1}_{\mathcal{E}^c}\rk
    \le \E\lk Z^s\mathbf{1}_{\mathcal{E}}\rk+\theta^s\lV Z\rV_{L_s}^s.
\end{equation}
Moreover, H\"older's inequality gives
\begin{equation}\label{eq:PZ2}
    \E\lk Z^s\mathbf{1}_{\mathcal{E}}\rk
    \le\lk \E Z^q\rk^{\frac{s}{q}} \lk \mathbb{P}\lk \mathcal{E}\rk\rk^{1-\frac{s}{q}}
    =
    \lV Z\rV_{L_q}^s\lk \mathbb{P}\lk \mathcal{E}\rk\rk^{1-\frac{s}{q}}.
\end{equation}
Combining~\eqref{eq:PZ1} and~\eqref{eq:PZ2}, we obtain
\begin{equation*}
    \lk 1-\theta^s\rk\lV Z\rV_{L_s}^s
    \le
    \lV Z\rV_{L_q}^s\lk \mathbb{P}\lk \mathcal{E}\rk\rk^{1-\frac{s}{q}}.
\end{equation*}
Rearranging yields~\eqref{eq:PZ0}.

\section{Auxiliary Proofs from Section~\ref{subsec:covariance}}

\subsection{\texorpdfstring{Proof of Equation~\eqref{eq:isotropic}}{Proof of Equation~(isotropic)}}\label{ap:realification}

Since $\E e^{\mathrm{i}\theta}=0$, we have
$\E\,\vb=\E e^{\mathrm{i}\theta}\cdot \E\,\va=\pmb{0}$.
Moreover, since $\va$ is isotropic, and by the independence of $\theta$ and $\va$, we obtain 
\begin{equation*}
\E\,\vb\vb^*=\E\,\va\va^*=\pmb{I}_n,\quad\E\,\vb\vb^{\top}
=\E\lk e^{2\mathrm{i}\theta}\va\va^{\top}\rk
=\E e^{2\mathrm{i}\theta}\cdot \E\,\va\va^{\top}
=\pmb{0}.
\end{equation*}
Comparing the real and imaginary parts in the two identities above, we obtain
\begin{equation*}
\E\,\Re\lk\vb\rk\Re\lk\vb\rk^{\top}
=
\E\,\Im\lk\vb\rk\Im\lk\vb\rk^{\top}
=
\frac{1}{2}\pmb{I}_n,
\qquad
\E\,\Re\lk\vb\rk\Im\lk\vb\rk^{\top}
=
\E\,\Im\lk\vb\rk\Re\lk\vb\rk^{\top}
=
\pmb{0}.
\end{equation*}
Therefore,
\begin{equation*}
\E\,\widetilde{\vb}\widetilde{\vb}^{\top}
=
\begin{pmatrix}
\E\,\Re\lk\vb\rk\Re\lk\vb\rk^{\top} 
& \E\,\Re\lk\vb\rk\Im\lk\vb\rk^{\top}\\
\E\,\Im\lk\vb\rk\Re\lk\vb\rk^{\top} 
& \E\,\Im\lk\vb\rk\Im\lk\vb\rk^{\top}
\end{pmatrix}
=
\frac{1}{2}\pmb{I}_{2n}.
\end{equation*}

Let $\x=(\x_1,\x_2)\in\mathbb{S}^{2n-1}$ with $\x_1,\x_2\in\mathbb{R}^n$, and set
$\z:=\x_1+\mathrm{i}\x_2\in\mathbb{S}_{\mathbb C}^{n-1}$.
Then
$\left\langle \widetilde{\vb},\x\right\rangle=\Re\lk\left\langle \vb,\z\right\rangle\rk$.
Hence,
\begin{equation*}
\begin{aligned}
\sup_{\x\in\mathbb S^{2n-1}}\lk\E\lv\left\langle\widetilde{\vb},\x\right\rangle\rv^p\rk^{1/p}
\le\sup_{\z\in\mathbb S_{\mathbb C}^{n-1}}\lk\E\lv\left\langle \vb,\z\right\rangle\rv^p\rk^{1/p}
=\sup_{\z\in\mathbb S_{\mathbb C}^{n-1}}\lk\E\lv\left\langle \va,\z\right\rangle\rv^p\rk^{1/p}
\le\tilde{\kappa}_p.
\end{aligned}
\end{equation*}

\subsection{\texorpdfstring{Proof of Equation~\eqref{eq:Raa}}{Proof of Equation~(Raa)}}\label{ap:Raa}
Since
\begin{equation*}
\vb\vb^*
=
\lk
\Re\lk\vb\rk\Re\lk\vb\rk^\top+
\Im\lk\vb\rk\Im\lk\vb\rk^\top
\rk+\mathrm{i}\lk\Im\lk\vb\rk\Re\lk\vb\rk^\top-
\Re\lk\vb\rk\Im\lk\vb\rk^\top
\rk,
\end{equation*}
the definition of $\mathcal{R}$ gives
\begin{equation*}
\mathcal R\lk \vb\vb^*\rk
=
\begin{pmatrix}
\Re\lk\vb\rk\Re\lk\vb\rk^\top+\Im\lk\vb\rk\Im\lk\vb\rk^\top &
\Re\lk\vb\rk\Im\lk\vb\rk^\top-\Im\lk\vb\rk\Re\lk\vb\rk^\top\\
\Im\lk\vb\rk\Re\lk\vb\rk^\top-\Re\lk\vb\rk\Im\lk\vb\rk^\top &
\Re\lk\vb\rk\Re\lk\vb\rk^\top+\Im\lk\vb\rk\Im\lk\vb\rk^\top
\end{pmatrix}.
\end{equation*}
Then, a direct calculation gives
\begin{equation*}
\widetilde{\vb}\widetilde{\vb}^{\top}+\widehat{\vb}\widehat{\vb}^{\top}
=
\begin{pmatrix}
\Re\lk\vb\rk\\
\Im\lk\vb\rk
\end{pmatrix}
\begin{pmatrix}
\Re\lk\vb\rk^\top & \Im\lk\vb\rk^\top
\end{pmatrix}+\begin{pmatrix}
-\Im\lk\vb\rk\\
\Re\lk\vb\rk
\end{pmatrix}
\begin{pmatrix}
-\Im\lk\vb\rk^\top & \Re\lk\vb\rk^\top
\end{pmatrix}=\mathcal{R}\lk \vb\vb^*\rk.
\end{equation*}

\subsection{\texorpdfstring{Proof of Equation~\eqref{eq:mu8}}{Proof of Equation~(mu8)}}\label{ap:mu8}

Let $X_i=\overline{a_i}x_i$. 
Then $\E\lz X_i\rz=0,\E\lv X_i\rv^2=\lv x_i\rv^2$, and 
$\E\lv X_i\rv^p=\lv x_i\rv^p\E\lv a_i\rv^p\le \alpha_p \lv x_i\rv^p$.
Applying Rosenthal's inequality in~\eqref{eq:rosen} yields, for any $\x\in\mathbb{S}_{\C}^{n-1}$,
\begin{equation*}
\begin{aligned}
\E\lv\lg\va,\x\rg\rv^p\lesssim_p \lk\sum_{i=1}^n \lv x_i\rv^2\rk^{p/2}+\sum_{i=1}^n \alpha_p \lv x_i\rv^p\lesssim_p 1+\alpha_p\sum_{i=1}^n \lv x_i\rv^p\lesssim_p \alpha_p.
\end{aligned}
\end{equation*}
Here, we used that for $p>2$,
$\sum_{i=1}^n \lv x_i\rv^p\le \lk\sum_{i=1}^n \lv x_i\rv^2\rk^{p/2}=1$
and $\alpha_p\ge1$.
Taking the supremum over $\x\in\mathbb{S}_{\C}^{n-1}$ gives~\eqref{eq:mu8}.

\subsection
{\texorpdfstring{Proof of Equation~\eqref{eq:diag}}{Proof of Equation~(diag)}}\label{ap:diag}

Note that $\lV\va\rV_{\ell_2}^2=\sum_{k=1}^n \lv a_k\rv^2$.
Thus the $(i,j)$-th entry of $\E\lV\va\rV_{\ell_2}^2\,\va\va^*$ equals
\begin{equation*}
\E\lz\lV\va\rV_{\ell_2}^2 a_i \overline{a_j}\rz
=\sum_{k=1}^n \E\lz \lv a_k\rv^2 a_i \overline{a_j}\rz.
\end{equation*}
If $i\neq j$, by independence and $\E a_i=\E a_j=0$, every term in the sum vanishes, hence the off-diagonal entries are zero.  
If $i=j$, then
\begin{equation*}
\E\lz\lV\va\rV_{\ell_2}^2 \lv a_i\rv^2\rz
=\E\lz \lv a_i\rv^4\rz+\sum_{k\neq i}\E\lz \lv a_k\rv^2\rz\E\lz \lv a_i\rv^2\rz
=\E\lz \lv a_i\rv^4\rz+\lk n-1\rk.
\end{equation*}
Therefore, we obtain~\eqref{eq:diag}.

\section{Proof of Fact~\ref{lm:max}}\label{ap:max}
We write
$\lV\va\rV_{\ell_2}^2:=n+\sum_{i=1}^n X_i$,
where $X_{i}:=\lv a_{i}\rv^2-1$.
As in~\eqref{eq:E2} and~\eqref{eq:Ep} in Section~\ref{subsec:quad}, we have
\begin{equation*}
\E\,\lv X_{i}\rv^2\le \alpha_4,\quad\text{and}\quad \E\,\lv X_{i}\rv^{p/2}\lesssim_p\alpha_{p}.
\end{equation*}
Hence, by Rosenthal's inequality in~\eqref{eq:rosen},
\begin{equation*}
\E\,\lv \sum_{i=1}^n X_{i}\rv^{p/2}
\lesssim_p \lk \alpha_4n\rk^{p/4}+\alpha_p n\lesssim_p \alpha_p n^{p/4}.
\end{equation*}
Now, let $u\ge 2n$. 
Then, by Markov's inequality,
\begin{equation*}
\begin{aligned}
\mathbb{P}\lk \lV\va\rV_{\ell_2}^2\ge u\rk
&=\mathbb{P}\lk \sum_{i=1}^n X_{i}\ge u-n\rk\le
\mathbb{P}\lk \lv \sum_{i=1}^n X_{i}\rv \ge \frac{u}{2}\rk\\
&\lesssim_p \frac{\E\,\lv \sum_{i=1}^n X_{i}\rv^{p/2}}{u^{p/2}}\lesssim_p\frac{\alpha_p n^{p/4}}{u^{p/2}}.
\end{aligned}
\end{equation*}
Finally, by the union bound,
\begin{equation*}
\mathbb{P}\lk \max_{1\le k\le m}\lV\va_k\rV_{\ell_2}^2\ge u\rk
\le\sum_{k=1}^m \mathbb{P}\lk \lV\va_k\rV_{\ell_2}^2\ge u\rk\lesssim_p\frac{\alpha_p mn^{p/4}}{u^{p/2}}.
\end{equation*}
Choosing $u= 2n+C\lk p\rk \alpha_p^{2/p}\sqrt{mn}$ with $C\lk p\rk$ sufficiently large completes the proof.

\normalem
\bibliographystyle{plain}

\bibliography{ref}

@article{candes2015phase,
  title={Phase retrieval via {W}irtinger flow: Theory and algorithms},
  author={Cand{\`e}s, Emmanuel J and Li, Xiaodong and Soltanolkotabi, Mahdi},
  journal={IEEE Transactions on Information Theory},
  volume={61},
  number={4},
  pages={1985--2007},
  year={2015},
  publisher={IEEE}
}

@article{candes2013phaselift,
  title={{PhaseLift}: Exact and stable signal recovery from magnitude measurements via convex programming},
  author={Cand{\`e}s, Emmanuel J and Strohmer, Thomas and Voroninski, Vladislav},
  journal={Communications on Pure and Applied Mathematics},
  volume={66},
  number={8},
  pages={1241--1274},
  year={2013},
  publisher={Wiley Online Library}
}

@article{huang2025stable,
  title={Stable Phase Retrieval: Optimal Rates in {Poisson} and Heavy-tailed Models},
  author={Huang, Gao and Li, Song and Needell, Deanna},
  journal={arXiv preprint arXiv:2510.00551},
  year={2025}
}

@article{candes2015phase1,
  title={Phase retrieval via matrix completion},
  author={Cand{\`e}s, Emmanuel J and Eldar, Yonina C and Strohmer, Thomas and Voroninski, Vladislav},
  journal={SIAM Review},
  volume={57},
  number={2},
  pages={225--251},
  year={2015},
  publisher={SIAM}
}

@article{kabanava2016stable,
  title={Stable low-rank matrix recovery via null space properties},
  author={Kabanava, Maryia and Kueng, Richard and Rauhut, Holger and Terstiege, Ulrich},
  journal={Information and Inference: A Journal of the IMA},
  volume={5},
  number={4},
  pages={405--441},
  year={2016},
  publisher={Oxford University Press}
}

@article{koltchinskii2015bounding,
  title={Bounding the smallest singular value of a random matrix without concentration},
  author={Koltchinskii, Vladimir and Mendelson, Shahar},
  journal={International Mathematics Research Notices},
  volume={2015},
  number={23},
  pages={12991--13008},
  year={2015},
  publisher={Oxford University Press}
}

@article{mendelson2015learning,
  title={Learning without concentration},
  author={Mendelson, Shahar},
  journal={Journal of the ACM (JACM)},
  volume={62},
  number={3},
  pages={1--25},
  year={2015},
  publisher={ACM New York, NY, USA}
}

@article{lecue2017sparse,
  title={Sparse recovery under weak moment assumptions},
  author={Lecu{\'e}, Guillaume and Mendelson, Shahar},
  journal={Journal of the European Mathematical Society},
  volume={19},
  number={3},
  pages={881--904},
  year={2017}
}

@article{kueng2017low,
  title={Low rank matrix recovery from rank one measurements},
  author={Kueng, Richard and Rauhut, Holger and Terstiege, Ulrich},
  journal={Applied and Computational Harmonic Analysis},
  volume={42},
  number={1},
  pages={88--116},
  year={2017},
  publisher={Elsevier}
}

@incollection{tropp2015convex,
  author    = {Tropp, Joel A.},
  title     = {Convex recovery of a structured signal from independent random linear measurements},
  booktitle = {Sampling Theory, a Renaissance: Compressive Sensing and Other Developments},
  pages     = {67--101},
  publisher = {Birkh{\"a}user},
  year      = {2015}
}

@article{krahmer2020complex,
  title={Complex phase retrieval from subgaussian measurements},
  author={Krahmer, Felix and St{\"o}ger, Dominik},
  journal={Journal of Fourier Analysis and Applications},
  volume={26},
  number={6},
  pages={89},
  year={2020},
  publisher={Springer}
}

@article{huang2025low,
  title={Low-rank {Toeplitz} matrix restoration: Descent cone analysis and structured random matrix},
  author={Huang, Gao and Li, Song},
  journal={IEEE Transactions on Information Theory},
  volume={71},
  number={5},
  pages={3950--3956},
  year={2025},
  publisher={IEEE}
}

@article{krahmer2021convex,
  title={On the convex geometry of blind deconvolution and matrix completion},
  author={Krahmer, Felix and St{\"o}ger, Dominik},
  journal={Communications on Pure and Applied Mathematics},
  volume={74},
  number={4},
  pages={790--832},
  year={2021},
  publisher={Wiley Online Library}
}

@article{tikhomirov2018sample,
  title={Sample covariance matrices of heavy-tailed distributions},
  author={Tikhomirov, Konstantin},
  journal={International Mathematics Research Notices},
  volume={2018},
  number={20},
  pages={6254--6289},
  year={2018},
  publisher={Oxford University Press}
}

@article{abdalla2024covariance,
  title   = {Covariance estimation: Optimal dimension-free guarantees for adversarial corruption and heavy tails},
  author  = {Abdalla, Pedro and Zhivotovskiy, Nikita},
  journal = {Journal of the European Mathematical Society},
  volume  = {28},
  number  = {4},
  pages   = {1809--1847},
  year    = {2026},
  doi     = {10.4171/JEMS/1505}
}

@article{jirak2025concentration,
  title={Concentration and moment inequalities for sums of independent heavy-tailed random matrices},
  author={Jirak, Moritz and Minsker, Stanislav and Shen, Yiqiu and Wahl, Martin},
  journal={Probability Theory and Related Fields},
  volume={194},
  pages={1917--1944},
  year={2026},
  doi={10.1007/s00440-025-01412-6},
  publisher={Springer}
}

@book{vershynin2018high,
  title     = {High-Dimensional Probability: An Introduction with Applications in Data Science},
  author    = {Vershynin, Roman},
  volume    = {47},
  year      = {2018},
  publisher = {Cambridge University Press}
}

@article{abdalla2022dictionary,
  title={Dictionary-sparse recovery from heavy-tailed measurements},
  author={Abdalla, Pedro and K{\"u}mmerle, Christian},
  journal={Information and Inference: A Journal of the IMA},
  volume={11},
  number={4},
  pages={1501--1526},
  year={2022},
  publisher={Oxford University Press}
}

@article{rudelson2013hanson,
  title={{Hanson-Wright} inequality and {sub-Gaussian} concentration},
  author={Rudelson, Mark and Vershynin, Roman},
  journal={Electronic Communications in Probability},
   volume={18},
  pages={1--9},
  year={2013}
}

@article{rosenthal1970subspaces,
  title={On the subspaces of $L_p$ ($p> 2$) spanned by sequences of independent random variables},
  author={Rosenthal, Haskell P},
  journal={Israel Journal of Mathematics},
  volume={8},
  number={3},
  pages={273--303},
  year={1970},
  publisher={Springer}
}

@article{rauhut2019low,
  title={Low-rank matrix recovery via rank one tight frame measurements},
  author={Rauhut, Holger and Terstiege, Ulrich},
  journal={Journal of Fourier Analysis and Applications},
  volume={25},
  number={2},
  pages={588--593},
  year={2019},
  publisher={Springer}
}

@article{gross2015partial,
  title={A partial derandomization of {PhaseLift} using spherical designs},
  author={Gross, David and Krahmer, Felix and Kueng, Richard},
  journal={Journal of Fourier Analysis and Applications},
  volume={21},
  number={2},
  pages={229--266},
  year={2015},
  publisher={Springer}
}

@article{gross2010quantum,
  title={Quantum state tomography via compressed sensing},
  author={Gross, David and Liu, Yi-Kai and Flammia, Steven T and Becker, Stephen and Eisert, Jens},
  journal={Physical Review Letters},
  volume={105},
  number={15},
  pages={150401},
  year={2010},
  publisher={APS}
}

@article{balan2015invertibility,
  title={Invertibility and robustness of phaseless reconstruction},
  author={Balan, Radu and Wang, Yang},
  journal={Applied and Computational Harmonic Analysis},
  volume={38},
  number={3},
  pages={469--488},
  year={2015},
  publisher={Elsevier}
}

@article{balan2016reconstruction,
  title={Reconstruction of signals from magnitudes of redundant representations: The complex case},
  author={Balan, Radu},
  journal={Foundations of Computational Mathematics},
  volume={16},
  number={3},
  pages={677--721},
  year={2016},
  publisher={Springer}
}

@article{eldar2014phase,
  title={Phase retrieval: Stability and recovery guarantees},
  author={Eldar, Yonina C and Mendelson, Shahar},
  journal={Applied and Computational Harmonic Analysis},
  volume={36},
  number={3},
  pages={473--494},
  year={2014},
  publisher={Elsevier}
}

@article{duchi2019solving,
  title={Solving (most) of a set of quadratic equalities: Composite optimization for robust phase retrieval},
  author={Duchi, John C and Ruan, Feng},
  journal={Information and Inference: A Journal of the IMA},
  volume={8},
  number={3},
  pages={471--529},
  year={2019},
  publisher={Oxford University Press}
}

@unpublished{abdalla_ramos_taylor_inprep,
  author = {Pedro Abdalla and Jo{\~a}o P. G. Ramos and Mitchell A. Taylor},
  note   = {In preparation},
  year   = {2026}
}

@article{candes2014solving,
  title={Solving quadratic equations via {PhaseLift} when there are about as many equations as unknowns},
  author={Cand{\`e}s, Emmanuel J and Li, Xiaodong},
  journal={Foundations of Computational Mathematics},
  volume={14},
  number={5},
  pages={1017--1026},
  year={2014},
  publisher={Springer}
}

@article{flammia2012quantum,
  title={Quantum tomography via compressed sensing: error bounds, sample complexity and efficient estimators},
  author={Flammia, Steven T and Gross, David and Liu, Yi-Kai and Eisert, Jens},
  journal={New Journal of Physics},
  volume={14},
  number={9},
  pages={095022},
  year={2012},
  publisher={IOP Publishing}
}

@article{ahmed2014compressive,
  title={Compressive multiplexing of correlated signals},
  author={Ahmed, Ali and Romberg, Justin},
  journal={IEEE Transactions on Information Theory},
  volume={61},
  number={1},
  pages={479--498},
  year={2014},
  publisher={IEEE}
}

@article{recht2010guaranteed,
  title={Guaranteed minimum-rank solutions of linear matrix equations via nuclear norm minimization},
  author={Recht, Benjamin and Fazel, Maryam and Parrilo, Pablo A},
  journal={SIAM Review},
  volume={52},
  number={3},
  pages={471--501},
  year={2010},
  publisher={SIAM}
}

@article{shechtman2015phase,
  title={Phase retrieval with application to optical imaging: a contemporary overview},
  author={Shechtman, Yoav and Eldar, Yonina C and Cohen, Oren and Chapman, Henry Nicholas and Miao, Jianwei and Segev, Mordechai},
  journal={IEEE Signal Processing Magazine},
  volume={32},
  number={3},
  pages={87--109},
  year={2015},
  publisher={IEEE}
}

@article{balan2009painless,
  title={Painless reconstruction from magnitudes of frame coefficients},
  author={Balan, Radu and Bodmann, Bernhard G and Casazza, Peter G and Edidin, Dan},
  journal={Journal of Fourier Analysis and Applications},
  volume={15},
  number={4},
  pages={488--501},
  year={2009},
  publisher={Springer}
}

@article{chandrasekaran2012convex,
  title={The convex geometry of linear inverse problems},
  author={Chandrasekaran, Venkat and Recht, Benjamin and Parrilo, Pablo A and Willsky, Alan S},
  journal={Foundations of Computational Mathematics},
  volume={12},
  number={6},
  pages={805--849},
  year={2012},
  publisher={Springer}
}

@article{demanet2014stable,
  title={Stable optimizationless recovery from phaseless linear measurements},
  author={Demanet, Laurent and Hand, Paul},
  journal={Journal of Fourier Analysis and Applications},
  volume={20},
  number={1},
  pages={199--221},
  year={2014},
  publisher={Springer}
}

@article{cai2015rop,
  author  = {T. Tony Cai and Anru Zhang},
  title   = {{ROP}: Matrix Recovery via Rank-One Projections},
  journal = {The Annals of Statistics},
  volume  = {43},
  number  = {1},
  pages   = {102--138},
  year    = {2015}
}

@article{chen2015exact,
  title={Exact and stable covariance estimation from quadratic sampling via convex programming},
  author={Chen, Yuxin and Chi, Yuejie and Goldsmith, Andrea J},
  journal={IEEE Transactions on Information Theory},
  volume={61},
  number={7},
  pages={4034--4059},
  year={2015},
  publisher={IEEE}
}

@article{krahmer2022robustness,
  title={On the robustness of noise-blind low-rank recovery from rank-one measurements},
  author={Krahmer, Felix and K{\"u}mmerle, Christian and Melnyk, Oleh},
  journal={Linear Algebra and its Applications},
  volume={652},
  pages={37--81},
  year={2022},
  publisher={Elsevier}
}

@article{li2021nonconvex,
  title={Nonconvex Matrix Factorization From Rank-One Measurements},
  author={Li, Yuanxin and Ma, Cong and Chen, Yuxin and Chi, Yuejie},
  journal={IEEE Transactions on Information Theory},
  volume={67},
  number={3},
  pages={1928--1950},
  year={2021},
  publisher={IEEE}
}

@article{foucart2019iterative,
  title={Iterative hard thresholding for low-rank recovery from rank-one projections},
  author={Foucart, Simon and Subramanian, Srinivas},
  journal={Linear Algebra and its Applications},
  volume={572},
  pages={117--134},
  year={2019},
  publisher={Elsevier}
}

@article{eisenmann2023riemannian,
  title={Riemannian thresholding methods for row-sparse and low-rank matrix recovery},
  author={Eisenmann, Henrik and Krahmer, Felix and Pfeffer, Max and Uschmajew, Andr{\'e}},
  journal={Numerical Algorithms},
  volume={93},
  number={2},
  pages={669--693},
  year={2023},
  publisher={Springer}
}

@article{mcrae2025nonconvex,
  author  = {McRae, Andrew D.},
  title   = {Phase retrieval and matrix sensing via benign and overparametrized nonconvex optimization},
  journal = {IEEE Transactions on Information Theory},
  volume  = {72},
  number  = {6},
  pages   = {4203--4220},
  year    = {2026},
  doi     = {10.1109/TIT.2026.3684748}
}

@inproceedings{qin2024general,
  title     = {A general algorithm for solving rank-one matrix sensing},
  author    = {Qin, Lianke and Song, Zhao and Zhang, Ruizhe},
  booktitle = {International Conference on Artificial Intelligence and Statistics},
  volume    = {238},
  series    = {Proceedings of Machine Learning Research},
  pages     = {757--765},
  year      = {2024},
  publisher = {PMLR}
}

@article{kim2024robust,
  title={Robust phase retrieval by alternating minimization},
  author={Kim, Seonho and Lee, Kiryung},
  journal={IEEE Transactions on Signal Processing},
  volume={73},
  pages={40--54},
  year={2024},
  publisher={IEEE}
}

@inproceedings{maunu2024acceleration,
  author    = {Tyler Maunu and Martin Molina-Fructuoso},
  title     = {Acceleration and Implicit Regularization in {Gaussian} Phase Retrieval},
  booktitle = {Proceedings of the 27th International Conference on Artificial Intelligence and Statistics},
  series    = {Proceedings of Machine Learning Research},
  volume    = {238},
  pages     = {4060--4068},
  year      = {2024},
  publisher = {PMLR}
}

@article{kingston2023optimizing,
  title={Optimizing nonconfigurable, transversely displaced masks for illumination patterns in classical ghost imaging},
  author={Kingston, Andrew M and Aminzadeh, Alaleh and Roberts, Lindon and Pelliccia, Daniele and Svalbe, Imants D and Paganin, David M},
  journal={Physical Review A},
  volume={107},
  number={2},
  pages={023524},
  year={2023},
  publisher={APS}
}

@book{kahane1985some,
  title     = {Some Random Series of Functions},
  author    = {Kahane, Jean-Pierre},
  volume    = {5},
  series    = {Cambridge Studies in Advanced Mathematics},
  edition   = {2},
  year      = {1985},
  publisher = {Cambridge University Press}
}

@article{gao2021phase,
  title={Phase retrieval for sub-{Gaussian} measurements},
  author={Gao, Bing and Liu, Haixia and Wang, Yang},
  journal={Applied and Computational Harmonic Analysis},
  volume={53},
  pages={95--115},
  year={2021},
  publisher={Elsevier}
}

@article{petrov2007lower,
  title={On lower bounds for tail probabilities},
  author={Petrov, Valentin V},
  journal={Journal of Statistical Planning and Inference},
  volume={137},
  number={8},
  pages={2703--2705},
  year={2007},
  publisher={Elsevier}
}

@article{gilles2025stable,
  title   = {Stable low-rank matrix recovery from 3-designs},
  author  = {Gilles, Timm},
  journal = {Applied and Computational Harmonic Analysis},
  volume  = {84},
  pages   = {101887},
  year    = {2026},
  doi     = {10.1016/j.acha.2026.101887}
}

@article{koren2009matrix,
  title={Matrix factorization techniques for recommender systems},
  author={Koren, Yehuda and Bell, Robert and Volinsky, Chris},
  journal={Computer},
  volume={42},
  number={8},
  pages={30--37},
  year={2009},
  publisher={IEEE}
}

@article{millane1990phase,
  title={Phase Retrieval in Crystallography and Optics},
  author={Millane, R. P.},
  journal={Journal of the Optical Society of America A},
  volume={7},
  number={3},
  pages={394--411},
  year={1990}
}

@article{huang2026robust,
  title   = {Robust outlier bound condition to phase retrieval with adversarial sparse outliers},
  author  = {Huang, Gao and Li, Song and Xu, Hang},
  journal = {Applied and Computational Harmonic Analysis},
  volume  = {80},
  pages   = {101819},
  year    = {2026},
  doi     = {10.1016/j.acha.2025.101819},
  publisher = {Elsevier}
}

@article{qu2020convolutional,
  title     = {Convolutional phase retrieval via gradient descent},
  author    = {Qu, Qing and Zhang, Yuqian and Eldar, Yonina C. and Wright, John},
  journal   = {IEEE Transactions on Information Theory},
  volume    = {66},
  number    = {3},
  pages     = {1785--1821},
  year      = {2020},
  doi       = {10.1109/TIT.2019.2950717},
  publisher = {IEEE}
}

@article{li2025truncated,
  title     = {Truncated amplitude flow with coded diffraction patterns},
  author    = {Li, Huiping and Li, Jiayi},
  journal   = {Inverse Problems},
  volume    = {41},
  number    = {1},
  pages     = {015002},
  year      = {2025},
  doi       = {10.1088/1361-6420/ad99f8},
  publisher = {IOP Publishing}
}

@article{aminzadeh2023mask,
  title   = {Mask design, fabrication, and experimental ghost imaging applications for patterned {X}-ray illumination},
  author  = {Aminzadeh, Alaleh and Roberts, Lindon and Young, Benjamin and Chiang, Cheng I. and Svalbe, Imants D. and Paganin, David M. and Kingston, Andrew M.},
  journal = {Optics Express},
  volume  = {31},
  number  = {15},
  pages   = {24328--24346},
  year    = {2023},
  doi     = {10.1364/OE.495024},
  publisher = {Optica Publishing Group}
}

@article{liu2010interior,
  title={Interior-point method for nuclear norm approximation with application to system identification},
  author={Liu, Zhang and Vandenberghe, Lieven},
  journal={SIAM Journal on Matrix Analysis and Applications},
  volume={31},
  number={3},
  pages={1235--1256},
  year={2010},
  publisher={SIAM}
}

\end{document}